\documentclass[11pt, leqno]{amsart}
\usepackage{amsmath, amssymb, latexsym}
\usepackage{mathrsfs}
\usepackage{enumerate}
\usepackage{color}
\usepackage[colorlinks=true, pdfstartview=FitV, linkcolor=blue,%
citecolor=blue, urlcolor=blue]{hyperref}
\usepackage[all, knot]{xy}
\xyoption{arc}
\usepackage{tikz-cd}
\usetikzlibrary{cd}
\usepackage{tikz}
\usepackage{amsmath}
\usetikzlibrary{cd}
\usetikzlibrary{arrows.meta}

\numberwithin{equation}{section}

\newtheorem{theorem}{Theorem}[section]
\newtheorem{corollary}[theorem]{Corollary}
\newtheorem{lemma}[theorem]{Lemma}
\newtheorem{proposition}[theorem]{Proposition}

\newtheorem{example}[theorem]{Example}

\theoremstyle{definition}
\newtheorem{definition}[theorem]{Definition}

\newcommand{\Hom}{\operatorname{Hom}}

\newcommand{\C}{\mathbf{C}}
\newcommand{\F}{\mathbf{F}}
\newcommand{\Q}{\mathbf{Q}}

\newcommand{\Z}{\mathbf{Z}}
\newcommand{\h}{\mathfrak{h}}
\newcommand{\g}{\mathfrak{g}}

\title[Borcherds-Bozec algebras]
{Borcherds-Bozec algebras, Root multiplicities \\ and the Schofield
construction}

\author[Seok-Jin Kang]
{Seok-Jin Kang$^{*}$}

\address{Research Institute of Computers, Information and Communication,
Pusan National University, 2 Busandaehak-ro, Pusan 46241, Korea}

         \email{soccerkang@hotmail.com}

\thanks{$^{*}$ This research was supported by Ministry of Culture, Sports and Tourism,
Korea Creative Content Agency in the Culture Technology Research and
Development Program 2017 and Basic Science Research Program of NRF
(Korea) under grant No. 2015R1D1A1A01059643.}


\begin{document}

\begin{abstract}

Using the twisted denominator identity, we derive a closed form root
multiplicity formula for all symmetrizable Borcherds-Bozec algebras
and discuss its applications including the case of Monster
Borcherds-Bozec algebra. In the second half of the paper, we provide
the Schofield construction of symmetric Borcherds-Bozec algebras.

\end{abstract}

\maketitle

\section*{Introduction}

\vskip 3mm

The aims of this paper are to investigate the root multiplicities of
Borcherds-Bozec algebras and to provide the Schofield construction
of these algebras. The Borcherds-Bozec algebras (actually, their
quantum deformation) arise as a natural algebraic structure relevant
to the theory of perverse sheaves on the representation varieties of
quivers with loops (see \cite{Bozec2014b, Bozec2014c, BSV2016} for
more details).

\vskip 3mm

To begin with, let us briefly clarify the terms of Kac-Moody
algebras, Borcherds algebras and Borcherds-Bozec algebras. The {\it
Kac-Moody algebras} were introduced independently by V. G. Kac and
R. V. Moody as a generalization of finite dimensional complex
semi-simple Lie algebras \cite{Kac68, Moody68}. In \cite{Serre}, one
can find J. P. Serre's presentation of finite dimensional complex
semi-simple Lie algebras by generators and relations associated with
positive definite Cartan matrices. What Kac and Moody did was to
remove the restriction of positive definiteness of Cartan matrices
and they were able to construct a new family of (mostly infinite
dimensional) Lie algebras, which may have roots with norms $\le 0$,
the {\it imaginary roots}. It turned out that, together with Kac's
discovery of the Weyl-Kac character formula \cite{Kac74},  Kac-Moody
algebras (affine Lie algebras in particular) have a lot of
significant applications to various branches of mathematics and
mathematical physics such as number theory, combinatorics and
statistical mechanics. Still, the generators of Kac-Moody algebras
are positive and negative simple root vectors having positive norms.

\vskip 3mm

In \cite{Bor88}, R. E. Borcherds gave a generalized version of
Kac-Moody algebras, the {\it Borcherds algebras}, associated with
{\it Borcherds-Cartan matrices}. These matrices are allowed to have
non-positive diagonal entries, which means the Borcherds algebras
may have simple roots with norms $\le 0$, the {\it imaginary simple
roots}. Thus we have a decomposition of index set for the simple
roots: $I=I^{\text{re}} \sqcup I^{\text{im}}$, where
$I^{\text{re}}=\{i \in I \mid a_{ii}=2\}$, the set of {\it real
indices}, and $I^{\text{im}} = \{ i\in I \mid a_{ii} \le 0\}$, the
set of {\it imaginary indices}. In the literature, the Borcherds
algebras are often called the {\it generalized Kac-Moody algebras}.
A special case of Borcherds algebras, the {\it Monster Lie algebra},
played an important role in Borcherds' proof of the Moonshine
Conjecture \cite{Bor92}. Nevertheless, the Borcherds algebras are
generated by positive and negative simple root vectors even though
they may have non-positive norms and multiplicities $>1$.

\vskip 3mm

The {\it Borcherds-Bozec algebras} form a further generalization of
Kac-Moody algebras. Their presentations are still associated with
Borcherds-Cartan matrices but the generators are {\it higher degree}
positive and negative  simple root vectors in the sense that their
degrees are integral multiples of simple roots. Accordingly, the
defining relations should be modified. In particular, when we deal
with quantum deformation, certain Drinfel'd-type relations are
included. We will denote by $I^{\infty} = (I^{\text{re}} \times
\{1\}) \cup (I^{\text{im}} \times \Z_{>0})$ the index set for higher
degree positive simple root vectors. Even though it seems more
appropriate to call them the Borcherds-Bozec-Kac-Moody algebras, for
simplicity, we will just use the term Borcherds-Bozec algebras.

\vskip 3mm

The first step toward the understanding of an algebraic object would
be to find a way of measuring its {\it size}. It can be shown that
every Borcherds-Bozec algebra has a decomposition into a direct sum
of homogeneous subspaces, the {\it root spaces}, and the dimensions
of such subspaces are called the {\it root multiplicities}. Thus one
gets naturally interested in (closed form or recursive form) the
{\it root multiplicity formulas}.

\vskip 3mm

In this paper, we give a closed form root multiplicity formula for
{\it all symmetrizable} Borcherds-Bozec algebras. To be more
precise, let $I$ be an index set and let $A=(a_{ij})_{i,j \in I}$ be
a Borcherds-Cartan matrix. As we have seen before, we have a
decomposition $I=I^{\text{re}} \sqcup I^{\text{im}}$. Take a finite
subset $J$ of $I^{\text{re}}$ so that the submatrix
$A_{J}=(a_{ij})_{i,j \in J}$ gives rise to a usual Kac-Moody
algebra. We denote by $\g$ the full Borcherds-Bozec algebra and
$\g_{0}^{(J)}$ the Kac-Moody algebra inside $\g$. In \cite{BSV2016},
T. Bozec, O. Schiffmann and E. Vasserot derived a character formula
for integrable highest weight $\g$-modules. Combining this with the
Weyl-Kac character formula, we prove the {\it twisted denominator
identity} for Borcehrds-Bozec algebras (Proposition
\ref{prop:twisted}), from which we obtain a closed form root
multiplicity formula (Theorem \ref{thm:mult}).

\vskip 3mm

Our formula has the following advantages. First, one can study the
structure of the Borcherds-Bozec algebra $\g$ as an integrable
representation of the Kac-Moody algebra $\g_{0}^{(J)}$. Second,
making an appropriate choice of $J$ can simplify the calculation and
provide a deeper understanding of root multiplicities. Finally,
different choices of $J$ would yield different expressions of the
same quantity, which naturally give rise to combinatorial
identities. We do not pursue these perspectives further in this
paper. Instead, we illustrate how one can apply our root
multiplicity formula in actual computation with the examples of rank
1 algebras, rank 2 algebras and the {\it Monster Borcherds-Bozec
algebra}. It should be pointed out that, in the symmetric case, the
root multiplicities are also given by the constant terms of
1-nilpotent Kac polynomials \cite{BSV2016}.

\vskip 3mm

The second half of this paper is devoted to the {\it Schofield
construction} of {\it symmetric} Borcherds-Bozec algebras (cf.
\cite{Sch}). Let $Q$ be a locally finite quiver with loops. Then one
can associate a symmetric Borcherds-Cartan matrix $A_{Q}$ and hence
a symmetric Borcherds-Bozec algebra $\g_{Q}$ as well. Using the
Euler characteristic of the projective variety of 1-nilpotent flags
in each representation of $Q$, we define a bilinear pairing
$$\langle \ , \ \rangle : {\mathscr E} \times Rep(Q)
\rightarrow \C,$$ where $\mathscr{E}$ is the free associative
algebra on the alphabet $\{S_{i,l} \mid (i,l) \in I^{\infty}\}$ and
$Rep(Q)$ denotes the category of representations of $Q$.

\vskip 3mm

Let ${\mathscr I}$ be the radical of the pairing $\langle \ , \
\rangle$ in ${\mathscr E}$. We show that ${\mathscr I}$ is both an
ideal and a co-ideal of ${\mathscr E}$ under the co-multiplication
$\Delta: {\mathscr E} \rightarrow {\mathscr E} \times {\mathscr E}$
given by $$S_{i,l} \mapsto S_{i,l} \otimes 1 + 1 \otimes S_{i,l} \ \
\text{for} \ (i,l) \in I^{\infty}.$$ Hence the quotient algebra
${\mathscr R} := {\mathscr E} \big/ {\mathscr I}$ becomes a
bi-algebra. Our main theorem (Theorem \ref{thm:main_c}) states that
${\mathscr R}$ is isomorphic to the positive part of the universal
enveloping algebra of $\g_{Q}$. Moreover, we show that the Lie
algebra ${\mathscr L}$ consisting of primitive elements in
${\mathscr R}$ is isomorphic to the positive part $\g_{Q}^{+}$ of
$\g_{Q}$.

\vskip 3mm

One of the key ingredients of our proof of the Serre relations in
${\mathscr R}$ is that there is a 1-1 correspondence between the set
of positive roots of $\g_{Q}$ and the set of dimension vectors of
1-nilpotent indecomposable representations of $Q$ over $\C$. The
argument of proving this statement is due to O. Schiffmann. We also
use Bozec's construction of $U^{+}(\g_{Q})$ in terms of certain
constructible functions on a Lagrangian subvariety of {\it strongly
semi-nilpotent representations} of $Q$, which is in turn based on
the theory of perverse sheaves and crystal bases for quantum
Borcherds-Bozec algebras \cite{Bozec2014b, Bozec2014c, Kas91}.

\vskip 3mm


We would like to express our sincere gratitude to Professor Bernard
Leclerc and Professor Olivier Schiffmann for many valuable
discussions and suggestions. Without their help, this work would not
have been materialized.

\vskip 10mm

\section{Borcherds-Bozec algebras}

\vskip 2mm

Let $I$ be an index set possibly countably infinite. An
integer-valued matrix $A=(a_{ij})_{i,j \in I}$ is called an {\it
even symmetrizable Borcherds-Cartan matrix} if it satisfies the
following conditions:

\begin{itemize}
\item[(i)] $a_{ii}=2, 0, -2, -4, ...$,

\item[(ii)] $a_{ij}\le 0$ for $i \neq j$,

\item[(iii)] $a_{ij}=0$ if and only if $a_{ji}=0$,

\item[(iv)] there exists a diagonal matrix $D=\text{diag} (s_{i} \in
\Z_{>0} \mid i \in I)$ such that $DA$ is symmetric.
\end{itemize}

\vskip 3mm

\noindent Set $I^{\text{re}}=\{i \in I \mid a_{ii}=2 \}$,
$I^{\text{im}}=\{i \in I \mid a_{ii} \le 0\}$ and
$I^{\text{iso}}=\{i \in I \mid a_{ii}=0 \}$.

\vskip 3mm

A {\it Borcherds-Cartan datum} consists of :

\ \ (a) an even symmetrizable Borcherds-Cartan matrix
$A=(a_{ij})_{i,j \in I}$,

\ \ (b) a free abelian group $P$, the {\it weight lattice},

\ \ (c) $\Pi=\{\alpha_{i} \in P  \mid i \in I \}$, the set of {\it
simple roots},

\ \ (d) $P^{\vee} := \Hom(P, \Z)$, the {\it dual weight lattice},

\ \ (e) $\Pi^{\vee}=\{h_i \in P^{\vee} \mid i \in I \}$, the set of
{\it simple coroots}

\vskip 2mm

\noindent satisfying the following conditions

\vskip 2mm

\begin{itemize}

\item[(i)] $\langle h_i, \alpha_j \rangle = a_{ij}$ for all $i,
j \in I$,

\item[(ii)] $\Pi$ is linearly independent,

\item[(iii)] for each $i \in I$, there exists an
element $\Lambda_{i} \in P$ such that $$\langle h_i, \Lambda_j
\rangle = \delta_{ij} \ \ \text{for all} \ i, j \in I.$$
\end{itemize}

\vskip 2mm

We would like to mention that given a Borcherds-Cartan matrix, such
a Borcherds-Cartan datum always exists. The $\Lambda_i$'s $(i \in
I)$ are called the {\it fundamental weights}.

\vskip 3mm

We denote by
$$P^{+}:=\{\lambda \in P \mid \langle h_i, \lambda \rangle \ge 0 \
\text{for all} \ i \in I \}$$ the set of {\it dominant integral
weights}. The free abelian group ${\mathsf Q}:= \bigoplus_{i \in I}
\Z \alpha_i$ is called the {\it root lattice}. Set ${\mathsf Q}_{+}
= \sum_{i\in I} \Z_{\ge 0} \alpha_i$. For $\beta = \sum k_i \alpha_i
\in {\mathsf Q}_{+}$, we define its {\it height} to be
$|\beta|:=\sum k_i$.

\vskip 3mm

Set $\mathfrak{h} = \C \otimes P^{\vee}$. For $\lambda, \mu \in
{\mathfrak h}^{*}$, we define a partial ordering by $\lambda \ge
\mu$ if and only if $\lambda - \mu \in {\mathsf Q}_{+}$. Since $A$
is symmetrizable, there exists a non-degenerate symmetric bilinear
form $( \ , \ )$ on ${\mathfrak h}^{*}$ satisfying
$$(\alpha_{i}, \lambda) = s_{i} \langle h_{i}, \lambda \rangle \quad
\text{for all} \ \ \lambda \in {\mathfrak h}^{*}.$$

\vskip 3mm

Let $I^{\infty}:= (I^{\text{re}} \times \{1\}) \cup (I^{\text{im}}
\times \Z_{>0})$. We will often write $i$ for $(i,1)$ $(i \in I^{\text{re}})$.

\vskip 3mm

\begin{definition}
The {\it Borcherds-Bozec algebra} $\g$ associated with a
Borcherds-Cartan datum $(A, P, \Pi, P^{\vee}, \Pi^{\vee})$ is the
Lie algebra over $\C$ generated by the elements $e_{il}$, $f_{il}$
$((i,l) \in I^{\infty})$ and $\h$ with defining relations
\begin{equation}\label{eq:bozec-alg}
\begin{aligned}
& [h, h']=0 \ \ \text{for} \ h, h' \in \h,\\
& [e_{ik}, f_{jl}] = k \, \delta_{ij} \, \delta_{kl} \, h_{i} \ \
\text{for} \ i,j \in I, k, l \in \Z_{>0},\\
& [h, e_{jl}]=l \, \langle h, \alpha_{j} \rangle \, e_{jl}, \quad
 [h, f_{jl}]= -l \, \langle h, \alpha_{j} \rangle \, f_{jl}, \\
& (\text{ad} e_{i})^{1-la_{ij}}(e_{jl}) = 0 \ \ \text{for} \ i\in
I^{\text{re}}, \, i \neq (j,l), \\
& (\text{ad} f_{i})^{1-la_{ij}}(f_{jl}) = 0 \ \ \text{for} \ i\in
I^{\text{re}}, \, i \neq (j,l), \\
& [e_{ik}, e_{jl}] = [f_{ik}, f_{jl}] =0 \ \ \text{for} \ a_{ij}=0.
\end{aligned}
\end{equation}
\end{definition}

Set $\text{deg}\, h =0$ for $h \in {\mathfrak h}$ and $\text{deg}\,
e_{il} = l \alpha_{i}$, $\text{deg}\, f_{il} = -l \alpha_{i}$ for
$(i,l) \in I^{\infty}$. Then $\g$ has a {\it root space
decomposition}
$$\g = \bigoplus_{\alpha \in \mathsf{Q}} \g_{\alpha},$$
where
$$\g_{\alpha} = \{x \in \g \mid [h,x] = \langle h, \alpha \rangle
\,x \ \ \text{for all} \ h \in \h\}.$$

\vskip 3mm

An element $\alpha \in \mathsf{Q} \setminus \{0\}$ is called a {\it
root} if $\g_{\alpha} \neq 0$. It can be shown that $\dim
\g_{\alpha} = \dim \g_{-\alpha}$, which is called the {\it root
multiplicity} of $\alpha$. Set $\Delta_{+}=\{\alpha \in
{\mathsf{Q}_{+}} \mid \g_{\alpha} \neq 0\}$ and $\Delta_{-} = -
\Delta_{+}$, the set of {\it positive} and {\it negative roots},
respectively.

\vskip 3mm

We say that a root $\alpha$ is {\it real} if $(\alpha, \alpha)
>0$ and {\it imaginary} if $(\alpha, \alpha) \le 0$. We denote by
$\Delta^{\text{re}}$ and $\Delta^{\text{im}}$, the set of real and
imaginary roots, respectively. Set
$\Delta_{+}^{\text{re}}=\Delta_{+}\cap \Delta^{\text{re}}$,
$\Delta_{+}^{\text{im}} = \Delta_{+} \cap \Delta^{\text{im}}$.

\vskip 3mm

For $i \in I^{\text{re}}$, we define the {\it simple reflection}
$r_{i} \in GL(\h^{*})$ by
$$r_{i} (\lambda) = \lambda - \langle h_{i}, \lambda \rangle
\alpha_{i} \ \ \text{for} \ \lambda \in \h^{*}.$$ The subgroup $W$
of $GL(\h^*)$ generated by $r_{i}$ $(i \in I^{\text{re}})$ is called
the {\it Weyl group} of $\g$. One can easily verify that the
symmetric bilinear form $(\ , \ )$ is $W$-invariant.

\vskip 7mm

\section{Representation theory}

\vskip 2mm

Let $\g$ be a Borcherds-Bozec algebra, let $U(\g)$ be its universal
enveloping algebra and let $M$ be a $\g$-module. We say that $M$ has
a {\it weight space decomposition} if
$$M = \bigoplus_{\mu \in \h^*} M_{\mu}, \ \ \text{where}
\ \ M_{\mu} = \{ m \in M \mid h \cdot m = \langle h, \mu \rangle m \
\text{for all} \ h \in \h \}.$$ We denote $\text{wt}(M):=\{\mu \in
\h^* \mid M_{\mu} \neq 0 \}$.

\vskip 3mm

A $\g$-module $V$ is called a {\it highest weight module with
highest weight $\lambda \in \h^*$} if there is a non-zero vector
$v_{\lambda} \in V$ such that
\begin{itemize}
\item[(i)] $V=U(\g) v_{\lambda}$,

\item[(ii)] $h \cdot v_{\lambda} = \langle h, \lambda \rangle
v_{\lambda}$ for all $h \in \h$,

\item[(iii)] $e_{il} \cdot v_{\lambda} =0$ for all $(i,l) \in
I^{\infty}$.
\end{itemize}

\vskip 2mm

\noindent Note that a highest weight module $V$ with highest weight
$\lambda$ has a weight space decomposition $V = \bigoplus_{\mu \le
\lambda} V_{\mu}$.

\vskip 3mm

For $\lambda \in \h^*$, let $J(\lambda)$ be the left ideal of
$U(\g)$ generated by the elements $h-\langle h, \lambda \rangle
\mathbf{1}$ for $h \in \h$ and $e_{il}$ for $(i,l) \in I^{\infty}$.
Then $M(\lambda):=U(\g) \big / J(\lambda)$ is a highest weight
$\g$-module with highest weight $\lambda$, called the {\it Verma
module}. It can be shown that $M(\lambda)$ has a unique maximal
submodule $R(\lambda)$ and every highest weight $\g$-module with
highest weight $\lambda$ is a homomorphic image of $M(\lambda)$. We
denote by $V(\lambda):=M(\lambda) \big / R(\lambda)$ the irreducible
quotient.

\vskip 3mm

\begin{proposition} \label{prop:hwmodule} \
{\rm  Let $\lambda \in P^{+}$ and $V(\lambda)=U(\g) v_{\lambda}$ be
the irreducible highest weight $\g$-module. Then we have
\begin{equation}\label{eq:hwmodule}
\begin{aligned}
& f_{i}^{\langle h_{i}, \lambda \rangle + 1} v_{\lambda} =0 \ \
\text{for} \ i \in I^{\text{re}}, \\
& f_{il} \, v_{\lambda} =0 \ \ \text{for} \ (i,l) \in I^{\infty} \
\text{with} \ \langle h_{i}, \lambda \rangle =0.
\end{aligned}
\end{equation} }
\end{proposition}

\vskip 3mm

\begin{proof}
If $f_{i}^{\langle h_{i}, \lambda \rangle +1} v_{\lambda} \neq 0$
for $i \in I^{\text{re}}$, then $f_{i}^{\langle h_{i}, \lambda
\rangle +1} v_{\lambda}$ would generate a submodule of $V(\lambda)$
with highest weight $\lambda - (\langle h_{i}, \lambda \rangle +1)
\alpha_{i} < \lambda$, a contradiction.

\vskip 3mm

Similarly, if $\langle h_{i}, \lambda \rangle =0$, $f_{il} \cdot
v_{\lambda}$ would generate a submodule with highest weight $\lambda
- l \alpha_{i} < \lambda$, which is also a contradiction.
\end{proof}

\vskip 3mm

\begin{proposition} \label{prop:imaginary} \ {\rm
Let $\lambda \in P^{+}$ and $V(\lambda)$ be the irreducible highest
weight $\g$-module. If $\mu \in \text{wt}(V(\lambda))$ and $i \in
I^{\text{im}}$, the following statements hold.

\vskip 3mm

(a) $\langle h_{i}, \mu \rangle \in \Z_{\ge 0}$.

\vskip 2mm

(b) If $\langle h_{i}, \mu \rangle =0$,  $V(\lambda)_{\mu + l
\alpha_{i}} =0$ for all $l \in \Z_{>0}$.

\vskip 2mm

(c) If $\langle h_{i}, \mu \rangle =0$, then $f_{il}\,
(V(\lambda)_{\mu})=0$ for all $l \in \Z_{>0}$.

\vskip 2mm

(d) If $\langle h_{i}, \mu \rangle \le -l a_{ii}$, then $e_{il}\,
(V(\lambda)_{\mu}) =0$ for all $l \in \Z_{>0}$.
 }

\begin{proof} \
Note that $\mu = \lambda - \beta$ for some $\beta \in
\mathsf{Q}_{+}$. Write $\beta = \sum_{j \in I} k_{j} \alpha_{j}$ for
$k_{j} \in \Z_{\ge 0}$. Then $\langle h_{i}, \mu \rangle = \langle
h_{i}, \lambda \rangle - \sum_{j\in I} k_{j} a_{ij} \ge 0$, which
proves (a).

\vskip 3mm

To prove (b), set $\text{supp}\, \beta := \{j \in I \mid k_{j} > 0
\}$. If $\langle h_{i}, \mu \rangle =0$, then $\langle h_{i},
\lambda \rangle =0$ and $a_{ij} =0$ whenever $j \in \text{supp}\,
\beta$. However, if $i \in \text{supp}\, \beta$, since $\langle
h_{i}, \lambda \rangle =0$, there is $j \in \text{supp}\, \beta$
with $j \neq i$ such that $a_{ij} \neq 0$, a contradiction. Hence $i
\notin \text{supp}\, \beta$ and $\mu + l \alpha_{i} \notin
\text{wt}(V(\lambda))$ for all $l \in \Z_{>0}$.

\vskip 3mm

For each $(i,l) \in I^{\infty}$, let $\g_{(i,l)}$ be the subalgebra
of $\g$ generated by $e_{il}$, $f_{il}$, $h_i$, which is isomorphic
to the Lie algebra $sl(2, \C)$ or the Heisenberg algebra. If
$\langle h_{i}, \mu \rangle =0$ and $v \in V(\lambda)_{\mu}$, we
have  $e_{il}\, v =0$ by (b). Then $v$ generates a 1-dimensional
$\g_{(i,l)}$-module, which implies $f_{il} \, v =0$. Thus (c) is
proved.

\vskip 3mm

Let us prove (d). Suppose $\langle h_{i}, \mu \rangle \le -l
a_{ii}$. If $\langle h_{i}, \mu \rangle < -l a_{ii}$, then $\langle
h_{i}, \mu + l \alpha_{i} \rangle < 0$ and $\mu + l \alpha_{i}
\notin \text{wt}(V(\lambda))$ by (a). If $\langle h_{i}, \mu \rangle
= -l a_{ii}$, write
$$\mu = \lambda - \beta = \lambda - k_{i} \alpha_{i} - \sum_{j \neq
i} k_{j} \alpha_{j}.$$ If $i \notin \text{supp}\, \beta$, we are
done. If $i \in \text{supp}\, \beta$ and $\mu + l \alpha_{i} \in
\text{wt}(V(\lambda))$, then $k_{i} \ge l$ and
$$\langle h_{i}, \mu + l \alpha_{i} \rangle = \langle h_{i}, \lambda
\rangle - (k_{i}-l) a_{ii} -\sum_{j \neq i} k_{j} a_{ij} =0.$$ It
follows that $\langle h_{i}, \lambda \rangle =0$, $(k_{i}-l)
a_{ii}=0$ and $a_{ij}=0$ for $j \in \text{supp}\, \beta$. Since
$\langle h_{i}, \lambda \rangle =0$, there exists $j \in
\text{supp}\, \beta$ such that $j \neq i$ and $a_{ij} \neq 0$, a
contradiction. Hence $\mu + l \alpha_{i} \notin
\text{wt}(V(\lambda))$.
\end{proof}
\end{proposition}

\vskip 3mm

\begin{definition}\label{def:Oint}
The {\it category ${\mathcal O}_{\text{int}}$} consists of
$\g$-modules $M$ such that
\begin{itemize}
\item [(i)] $M$ has a weight space decomposition $M=\bigoplus_{\mu
\in P} M_{\mu}$ with $\dim_{\C} M_{\mu} < \infty$ for all $\mu \in
P$.

\item[(ii)] there exist finitely many weights $\lambda_1, \ldots,
\lambda_s \in P$ such that $$\text{wt}(M) \subset \bigcup_{j=1}^s
(\lambda_j -{\mathsf Q}_{+}),$$

\item[(iii)] if $i \in I^{\text{re}}$, $f_{i}$ is locally nilpotent
on $M$,

\item[(iv)] if $i \in I^{\text{im}}$, we have $\langle h_i, \mu
\rangle \ge 0$ for all $\mu \in \text{wt}(M)$,

\item[(v)] if $i \in I^{\text{im}}$ and $\langle h_{i}, \mu \rangle
=0$, then $f_{il}(M_{\mu})=0$ for all $l \in \Z_{>0}$,

\item[(vi)] if $i \in I^{\text{im}}$ and $\langle h_i, \mu \rangle
\le - l a_{ii}$, then $e_{il}(M_{\mu})=0$ for all $l \in \Z_{>0}$.
\end{itemize}
\end{definition}

\vskip 3mm

Note that in this definition, $f_{il}$'s are not necessarily locally
nilpotent.

\vskip 3mm

By Proposition \ref{prop:hwmodule} and Proposition
\ref{prop:imaginary}, the irreducible highest weight $\g$-module
$V(\lambda)$ with $\lambda \in P^{+}$ belongs to the category
${\mathcal O}_{\text{int}}$. More generally, let
$V=U(\g)v_{\lambda}$ be a highest weight $\g$-module with $\lambda
\in P^{+}$ and assume that $V$ satisfies the condition
\eqref{eq:hwmodule}. Then using the results in \cite[Lemma 3.4,
Lemma 3.5]{Kac90}, one can show that $V$ lies in the category
${\mathcal O}_{\text{int}}$.

\vskip 3mm

Let $\lambda \in P^{+}$ and let $F_{\lambda}$ be the set of elements
of the form $s= \sum_{k=1}^r s_k \alpha_{i_k}$ $(r \ge 0)$ such that
\begin{itemize}
\item[(i)] $i_{k} \in I^{\text{im}}$, $s_{k} \in \Z_{>0}$ for all $1 \le
k \le r$,

\item[(ii)] $(\alpha_{i_k}, \alpha_{i_l}) =0$ for all $1 \le k,l \le
r $,

\item[(iii)] $(\alpha_{i_k}, \lambda) =0$ for all $1 \le k \le r$.
\end{itemize}

\vskip 3mm

\noindent When $r=0$, we understand $s=0$.

\vskip 3mm

For $s=\sum s_k \alpha_{i_k} \in F_{\lambda}$, we define
\begin{equation} \label{eq:sign}
\begin{aligned}
 d_{i}(s) & = \begin{cases} \# \{k \mid i_k =i \} & \text{if} \ i
\notin I^{\text{iso}}, \\
\sum_{i_k=i} s_k & \text{if} \ i \in I^{\text{iso}},
\end{cases} \\
\epsilon(s) & = \prod_{i \notin I^{\text{iso}}} (-1)^{d_i(s)}
\prod_{i \in I^{\text{iso}}} \phi(d_{i}(s))\\
& = (-1)^{\#(\text{supp}(s) \cap I \setminus I^{\text{iso}})}
\prod_{i \in I^{\text{iso}}} \phi(d_{i}(s)),
\end{aligned}
\end{equation}
where $\prod_{k=1}^{\infty} (1-q^k) = \sum_{n\ge 0} \phi(n) q^n.$

\vskip 3mm

We define
\begin{equation} \label{eq:slambda}
S_{\lambda} = \sum_{s \in F_{\lambda}} \epsilon(s) e^{-s}.
\end{equation}

\vskip 3mm

\noindent
{\bf Remark.} The subalgebra of $\g$ generated by the
elements $e_{i,1}$, $f_{i,1}$ $(i \in I)$ and $\h$ is the Borcherds
algebra associated with the given Borcherds-Cartan datum $(A, P,
\Pi, P^{\vee}, \Pi^{\vee})$. In this case, all $s_k=1$ and hence we
have
\begin{equation*}
\begin{aligned}
& d_{i}(s) = \# \{k \mid i_k = i \} \ \ \text{for all} \ i \in
I^{\text{im}},\\
& \epsilon(s) = \prod_{i \in I^{\text{im}}} (-1)^{d_i(s)} =
(-1)^{|s|}. \end{aligned}
\end{equation*}

\vskip 3mm

\begin{example} \label{ex:rank1} \hfill

\vskip 2mm

{\rm  Let $i \in I^{\text{im}}$ and $\g_{(i)}$ be the
Borcherds-Bozec algebra associated with $A=(a_{ii})$.

\vskip 3mm

(a) If $a_{ii}=0$, then $\g_{(i)}$ is the Lie algebra generated by
$h_i$, $e_{il}$, $f_{il}$ $(l \ge 1)$ with defining relations
\begin{equation*}
[h_i, e_{il}] = [h_i, f_{il}]=0, \quad [e_{ik}, f_{il}] = k
\delta_{kl} h_i, \quad [e_{ik}, e_{il}] = [f_{ik}, f_{il}]=0
\end{equation*}
for all $k, l \ge 1$. Hence $\g_{(i)}$ is the full Heisenberg
algebra.

\vskip 3mm

If $\langle h_i, \lambda \rangle >0$, then $F_{\lambda}=\{0\}$ and
$S_{\lambda}=1$.

\vskip 3mm

If $\langle h_i, \lambda \rangle =0$, then $F_{0}=\{ l \alpha_{i}
\mid l \ge 0 \}$ and $d_{i}(l \alpha_{i}) =l$ for all $l \ge 0$.
Hence
$$S_{0}=\sum_{l \ge 0} \epsilon(l \alpha_i) \, e^{-l\alpha_i} = \sum_{l
\ge 0} \phi(l) \, e^{-l \alpha_i} = \prod_{k=1}^{\infty} (1- e^{-k
\alpha_{i}}).$$ Note that $U(\g_{(i)}^{-}) \cong \C[f_{il} \mid l
\ge 1]$, which implies
$$\text{ch} U(\g_{(i)}^{-}) = \prod_{k=1}^{\infty} \dfrac{1}{1-e^{-k
\alpha_i}} = S_{0}^{-1}.$$

\vskip 3mm

(b) If $a_{ii} <0$, we have
\begin{equation*}
 [h_i, e_{il}]=l a_{ii} \, e_{il}, \quad
 [h_i, f_{il}]= - l a_{ii} \, f_{il},\quad
 [e_{ik}, f_{il}]=k \delta_{kl} h_i
\end{equation*}
and there are no relations among $e_{il}$'s nor among $f_{il}$'s.
Hence $\g_{(i)}^{+}$ (resp. $\g_{(i)}^{-}$) is the free Lie algebra
on $V^{+}= \bigoplus_{l \ge 1} \C e_{il}$ (resp. $V^{-} =
\bigoplus_{l \ge 1} \C f_{il}$).

\vskip 3mm

If $\langle h_{i}, \lambda \rangle >0$, then $F_{\lambda}=\{0\}$ and
$S_{\lambda}=1$.

\vskip 3mm

If $\langle h_i, \lambda \rangle =0$, we have
\begin{equation*}
 F_{0} = \{ l \alpha_{i} \mid l \ge 0 \}, \quad
 d_i (l \alpha_i) = \begin{cases} 0 & \text{if} \ l=0, \\
1 & \text{if} \ l \ge 1,
\end{cases}\quad
 \epsilon(l \alpha_i) = \begin{cases} 1 & \text{if} \ l=0, \\
-1 & \text{if} \ l \ge 1,
\end{cases}
\end{equation*}
which implies
$$S_{0} = 1 - (e^{-\alpha_i} + e^{-2 \alpha_i} + \cdots ) = 1 -
e^{-\alpha_i} \dfrac{1}{1-e^{-\alpha_i}} = \dfrac{1-2
e^{-\alpha_i}}{1-e^{-\alpha_i}}.$$

\vskip 2mm

On the other hand, $U(\g_{(i)}^{-}) \cong \C \langle f_{il} \mid l \ge 1
\rangle$, the free  associative algebra on $V^{-} = \bigoplus_{l\ge
1} \C f_{il}$, the tensor algebra on $V^{-}$. Since
$$\text{ch} V^{-} = \dfrac{e^{-\alpha_i}}{1-e^{-\alpha_i}},$$
we have
$$\text{ch} U(\g_{(i)}^{-}) = \dfrac{1}{1- \text{ch} V^{-}} =
\dfrac{1}{1-\dfrac{e^{-\alpha_i}}{1-e^{-\alpha_i}}}=
\dfrac{1-e^{-\alpha_i}}{1-2 e^{-\alpha_i}} = S_{0}^{-1},$$ again. }
\qed
\end{example}

\vskip 3mm

Let $\rho$ be a linear functional on $\h$ such that $\langle h_i,
\rho \rangle =1$ for all $i\in I$. For each $w \in W$, we denote by
$l(w)$ the length of $w$ and set $\epsilon(w)= (-1)^{l(w)}$.

\vskip 3mm

\begin{theorem} \label{thm:character} \cite{BSV2016} \
{\rm Let $V=U(\g) v_{\lambda}$ be a highest weight $\g$-module with
highest weight $\lambda \in P^{+}$. If $V$ satisfies the condition
\eqref{eq:hwmodule}, then the character of $V$ is given by the
following formula\,:

\begin{equation} \label{eq:character}
\begin{aligned}
\text{ch} V &= \dfrac{\sum_{w \in W} \epsilon (w) e^{w(\lambda +
\rho)-\rho} w (S_{\lambda})}{\prod_{\alpha \in \Delta_{+}} (1-
e^{-\alpha})^{\dim \g_{\alpha}}} \\
&= \dfrac{\sum_{w \in W} \sum_{s \in F_{\lambda}} \epsilon(w)
\epsilon(s) e^{w(\lambda + \rho -s) - \rho}}{\prod_{\alpha \in
\Delta_{+}} (1-e^{-\alpha})^{\dim \g_{\alpha}}}.
\end{aligned}
\end{equation}
}
\end{theorem}

\vskip 2mm

\begin{proof}
The argument given in \cite[Theorem 6.1]{BSV2016} works for any
highest weight $\g$-module with highest weight $\lambda \in P^{+}$
satisfying the condition \eqref{eq:hwmodule}.
\end{proof}

\vskip 3mm

In particular, when $\lambda =0$, we obtain the {\it denominator
identity}

\begin{equation} \label{eq:denom}
\begin{aligned}
\prod_{\alpha \in \Delta_{+}} & (1-e^{-\alpha})^{\dim \g_{\alpha}} =
\sum_{w\in W} \epsilon(w) e^{w \rho - \rho} w(S_{0}) \\
& = \sum_{w \in W} \sum_{s \in F_{0}} \epsilon(w) \epsilon(s)
e^{w(\rho -s) - \rho}.
\end{aligned}
\end{equation}

\vskip 3mm

In Borcherds algebra case, we have
\begin{equation*}
\text{ch} V = \dfrac{\sum_{w \in W} \sum_{s \in F_{\lambda}}
(-1)^{l(w)+|s|}  e^{w(\lambda + \rho -s) - \rho}}{\prod_{\alpha \in
\Delta_{+}} (1-e^{-\alpha})^{\dim \g_{\alpha}}},
\end{equation*}
which yields
\begin{equation*}
\prod_{\alpha \in \Delta_{+}} (1-e^{-\alpha})^{\dim \g_{\alpha}}
=\sum_{w \in W} \sum_{s \in F_{0}} (-1)^{l(w)+|s|} e^{w(\rho -s) -
\rho}.
\end{equation*}

\vskip 3mm

Now we will show that every $\g$-module in the category ${\mathcal
O}_{\text{int}}$ is completely reducible.

\vskip 3mm

\begin{proposition} \label{prop:Oint} \hfill

\vskip 2mm

{\rm
\begin{enumerate}
\item[(a)] Every highest weight $\g$-module with highest weight $\lambda
\in P^{+}$ satisfying \eqref{eq:hwmodule} is isomorphic to
$V(\lambda)$.

\vskip 2mm

\item[(b)] If $V(\lambda)$ belongs to the category ${\mathcal
O}_{\text{int}}$, then $\lambda \in P^{+}$.

\vskip 2mm

\item[(c)] Every simple object in the category ${\mathcal O}_{\text{int}}$
is isomorphic to $V(\lambda)$ for some $\lambda \in P^{+}$.

\vskip 2mm

\item[(d)] The category ${\mathcal O}_{\text{int}}$ is semisimple.
\end{enumerate}}
\end{proposition}

\begin{proof} \ (a) follows from the character formula.

\vskip 2mm

\begin{enumerate}
\item[(b)] If $i \in I^{\text{re}}$, then by the standard $sl(2,
\C)$-theory, $\langle h_i, \lambda \rangle \ge 0$. If $i \in
I^{\text{im}}$, by the definition of ${\mathcal O}_{\text{int}}$, we
have $\langle h_i, \lambda \rangle \ge 0$.

\vskip 2mm

\item[(c)] Every simple object in ${\mathcal O}_{\text{int}}$is a highest
weight $\g$-module because any maximal weight vector would generate
a highest weight submodule. Then the statement (b) implies (c).

\vskip 2mm

\item[(d)] Thanks to Theorem \ref{thm:character}, the same argument in
\cite[Theorem 3.5.4]{HK02} or \cite[Theorem 3.7]{JKK05} would prove
our claim.
\end{enumerate}
\end{proof}

\vskip 3mm

\begin{example} \label{ex:char-rank2}

{\rm

\vskip 2mm

Let $i \in I^{\text{im}}$ and $\g_{(i)}$ be the Borcherds-Bozec
algebra associated with $A=(a_{ii})$. If $\langle h_i, \lambda
\rangle =0$, then the case is trivial; i.e., $V(\lambda)=\C
v_{\lambda}$ and $\text{ch} V(\lambda) = e^{\lambda}$. Thus we
assume that $\langle h_i, \lambda \rangle >0$.

\vskip 3mm

(a) If $a_{ii}=0$, then $W=\{1\}$ and $S_{\lambda}=1$. Hence
\begin{equation*}
\text{ch} V(\lambda)=\dfrac{e^{\lambda}}{\prod_{k=1}^{\infty}
(1-e^{-k\alpha_i})^{\dim \g_{k\alpha_i}}}.
\end{equation*}

\vskip 2mm

By the denominator identity and Example \ref{ex:rank1} (a), we have

\begin{equation*}
\prod_{k=1}^{\infty} (1-e^{-k\alpha_i})^{\dim \g_{k\alpha_i}} =
S_{0} = \prod_{k=1}^{\infty} (1-e^{-k\alpha_i}).
\end{equation*}

\vskip 3mm

It follows that $\dim \g_{k \alpha_i} =1$ for all $k \ge 1$ and
$$\text{ch} V(\lambda) = \dfrac{e^{\lambda}}{\prod_{k=1}^{\infty}(1-e^{-k
\alpha_i})}.$$

\vskip 3mm

(b) If $a_{ii} <0$, then again $W=\{1\}$, $S_{\lambda}=1$ and the
denominator identity combined with Example \ref{ex:rank1}(b) gives

\begin{equation*}
\text{ch} V(\lambda) = \dfrac{e^{\lambda}}{\prod_{k=1}^{\infty}
(1-e^{-k \alpha_i})^{\dim \g_{k \alpha_{i}}}} = e^{\lambda}
S_{0}^{-1} = \dfrac{e^{\lambda} (1-e^{-\alpha_i})}{1-2
e^{-\alpha_i}}.
\end{equation*}

\vskip 3mm

In each case, we have $\text{ch} V(\lambda) = e^{\lambda} \text{ch}
U(\g_{(i)}^{-})$. Thus when $\langle h_i, \lambda \rangle >0$, the
irreducible $\g_{(i)}$-module $V(\lambda)$ is isomorphic to the
Verma module $M(\lambda)$. Therefore, for each $i \in
I^{\text{im}}$, every $\g_{(i)}$-module $M$ in the category
${\mathcal O}_{\text{int}}$ is isomorphic to a direct sum of
1-dimensional trivial modules and (infinite-dimensional) Verma
modules. } \qed
\end{example}

\vskip 7mm

\section{Root multiplicity formula}

\vskip 3mm

Recall the denominator identity \eqref{eq:denom}:

\begin{equation*}
\prod_{\alpha \in \Delta_{+}} (1-e^{-\alpha})^{\dim \g_{\alpha}}=
\sum_{w \in W} \sum_{s \in F_{0}} \epsilon(w) \epsilon(s) e^{w(\rho
-s) - \rho},
\end{equation*}
where $F_{0}$ is the set of elements of the form $s= \sum_{k=1}^r
s_k \alpha_{i_k}$ $(r \ge 0)$ such that
\begin{itemize}
\item[(i)] $i_{k} \in I^{\text{im}}$, $s_{k} \in \Z_{>0}$ for all $1 \le
k \le r$,

\item[(ii)] $(\alpha_{i_k}, \alpha_{i_l}) =0$ for all $1 \le k,l \le
r $.
\end{itemize}

\vskip 3mm

Take a finite subset $J \subset I^{\text{re}}$ and set
$$\Delta^{J} = \Delta \cap (\sum_{j \in J} \Z \alpha_j),
\quad \Delta_{\pm}^{J}= \Delta_{\pm} \cap \Delta^{J}, \quad
\Delta_{\pm}(J) = \Delta_{\pm} \setminus \Delta_{\pm}^{J},$$
$$W^{J}= \langle r_{j} \mid j \in J \rangle, \quad W(J)=\{w \in W
\mid w \Delta_{-} \cap \Delta_{+} \subset \Delta_{+}(J) \}.$$

Note that $W(J)$ is a set of right coset representatives of $W^{J}$
in $W$. Hence when $w=w' r_j$ with $l(w) = l(w')+1$, $w \in W(J)$ if
and only if $w' \in W(J)$ and $w'(\alpha_j) \in \Delta_{+}(J)$. In this way, one can construct $W(J)$ inductively.

\vskip 3mm

Let
$$\g_{0}^{(J)}:=\h \oplus (\sum_{\alpha \in \Delta^{J}}
\g_{\alpha}), \quad \g_{\pm}^{(J)} = \sum_{\alpha \in
\Delta_{\pm}^{(J)}} \g_{\alpha}.$$ Then we have a {\it twisted
triangular decomposition}
$$\g = \g_{-}^{(J)} \oplus \g_{0}^{(J)} \oplus \g_{+}^{(J)}.$$

\vskip 3mm

Let $P_{J}^{+} = \{ \lambda \in P \mid \langle h_j, \lambda \rangle
\ge 0 \ \, \text{for all} \ j \in J \}$ and let $V_{J}(\lambda)$ be
the irreducible highest weight $\g_{0}^{(J)}$-module with highest
weight $\lambda \in P_{J}^{+}$. Then by the Weyl-Kac character
formula \cite{Kac74, Kac90}, we have
$$\text{ch} V_{J}(\lambda) = \dfrac{\sum_{w \in W^{J}} \epsilon(w)
e^{w(\lambda + \rho) - \rho}}{\prod_{\alpha \in \Delta_{+}^{J}}
(1-e^{-\alpha})^{\dim \g_{\alpha}}}.$$

Using this and the denominator identity \eqref{eq:denom}, we obtain
the following {\it twisted denominator identity} for Borcherds-Bozec
algebras, which will be the corner-stone for our root multiplicity
formula.

\vskip 3mm

\begin{proposition} \label{prop:twisted}

\vskip 2mm

{\rm
\begin{equation} \label{eq:twisted}
\prod_{\alpha \in \Delta^{+}(J)} (1-e^{-\alpha})^{\dim \g_{\alpha}}
=\sum_{w \in W(J)} \sum_{s \in F_{0}} \epsilon(w) \epsilon(s)
\text{ch} V_{J}(w(\rho -s) -\rho).
\end{equation}
}
\end{proposition}

\begin{proof} \
The right
-hand side is equal to
\begin{equation*}
\begin{aligned}
& \sum_{w \in W(J)} \sum_{s \in F_{0}} \epsilon(w) \epsilon(s)
\dfrac{\sum_{w' \in W^{J}} \epsilon(w') e^{w'(w(\rho-s)-\rho + \rho)
-\rho}} {\prod_{\alpha \in \Delta_{+}^{J}} (1-e^{-\alpha})^{\dim
\g_{\alpha}}} \\
& = \dfrac{\sum_{w \in W(J)} \sum_{s \in F_{0}} \sum_{w' \in W^{J}}
\epsilon(w) \epsilon(w') \epsilon(s) e^{w' w (\rho - s) -
\rho}}{\prod_{\alpha \in \Delta_{+}^{J}} (1-e^{-\alpha})^{\dim
\g_{\alpha}}} \\
& = \dfrac{\sum_{w \in W} \sum_{s \in F_{0}} \epsilon(w) \epsilon(s)
e^{w(\rho-s)-\rho}}{\prod_{\alpha \in \Delta_{+}^{J}}
(1-e^{-\alpha})^{\dim
\g_{\alpha}}} \\
&= \dfrac{\prod_{\alpha \in \Delta_{+}} (1-e^{-\alpha})^{\dim
\g_{\alpha}}}{\prod_{\alpha \in \Delta_{+}^{J}}
(1-e^{-\alpha})^{\dim \g_{\alpha}}} =\prod_{\alpha \in
\Delta_{+}(J)} (1-e^{-\alpha})^{\dim \g_{\alpha}}
\end{aligned}
\end{equation*}
as desired.
\end{proof}

\vskip 3mm

For $s= \sum s_k \alpha_{i_k} \in F_{0}$, set
\begin{equation*}
s^{+}:=\sum_{i_k \notin I^{\text{iso}}} s_k \alpha_{i_k}, \quad s^0
:=\sum_{i_k \in I^{\text{iso}}} s_k \alpha_{i_k}, \quad d(s):=
\sum_{i \in I^{\text{im}}} d_i(s).
\end{equation*}
Then we have
$$s = s^{+} + s^{0}, \quad d(s) = d(s^{+}) + d(s^{0}), \quad
\epsilon(s) = \epsilon(s^{+}) \, \epsilon(s^{0}).$$
Here, we
understand $d(0)=0$, $\epsilon(0)=1$.

\vskip 3mm

Define
\begin{equation} \label{eq:spaceV}
V^{(J)} = \bigoplus_{\substack{w \in W(J)\\ s \in F_{0}\\
l(w)+d(s)>0}} (-1)^{l(w) + d(s^{+})+1} \, \epsilon(s^{0}) \,
V_{J}(w(\rho-s)-\rho),
\end{equation}
a virtual direct sum of $\g_{0}^{(J)}$-modules. Then the twisted
denominator identity can be written as

\begin{equation} \label{eq:twisted2}
\prod_{\alpha \in \Delta_{+}^{(J)}} (1-e^{-\alpha})^{\dim
\g_{\alpha}} = 1 - \text{ch} V^{(J)}.
\end{equation}

\vskip 3mm

For each $\mu \in P$, we denote the {\it virtual dimension} of
$V^{(J)}_{\mu}$ by
\begin{equation*}
\begin{aligned}
& d_{\mu} := \dim V^{(J)}_{\mu} \\
& = \sum_{\substack{w \in W(J) \\ s \in F_{0} \\ l(w) + d(s) >0}}
(-1)^{l(w) + d(s^+) +1} \epsilon(s^0)\, \dim V_{J}(w(\rho -s) -
\rho)_{\mu}.
\end{aligned}
\end{equation*}

\vskip 2mm

Set $\text{wt}(V^{(J)}) = \{ \mu \in P \mid d_{\mu} \neq 0 \}$ and
give an enumeration $(\mu_1, \mu_2, \ldots)$ of $\text{wt}
(V^{(J)})$. For each $\mu \in P$, let ${\mathscr P}^{(J)}(\mu)$ be
the set of partitions of $\mu$ into a sum of $\mu_{i}$'s\,:

\begin{equation} \label{eq:partition}
{\mathscr P}^{(J)}(\mu) = \{ s =(s_i)_{i \ge 1} \mid s_i \in \Z_{\ge
0}, \ \sum s_i \mu_i = \mu \}.
\end{equation}

\vskip 2mm

For a partition $s=(s_{i}) \in {\mathscr P}^{(J)}(\mu)$, we write
$|s|:=\sum s_i$, $s! = \prod s_{i}!$ and define the {\it Witt
partition function of $\mu$} by

\begin{equation} \label{eq:witt}
W^{(J)}(\mu) = \sum_{s \in {\mathscr P}^{(J)}(\mu)}
\dfrac{(|s|-1)!}{s!} \prod (d_{\mu_i})^{s_i}.
\end{equation}

\vskip 3mm

We now derive our root multiplicity formula for Borcherds-Bozec
algebras (cf. \cite{Kang94a}).

\vskip 2mm

\begin{theorem} \label{thm:mult} \
{\rm For any root $\alpha \in \Delta_{+}^{(J)}$, we have
\begin{equation} \label{eq:mult}
\dim \g_{\alpha} = \sum_{d | \alpha} \frac{1}{d}\, \mu(d)\,
W^{(J)}(\alpha \big/ d),
\end{equation}
where $\mu$ denotes the M\"obius function.}
\end{theorem}

\vskip 2mm

\begin{proof} \ Note that $\text{wt}(V^{(J)}) \subset -{\mathsf
Q}_{+}$. We write
$$\text{ch} (V^{(J)}) = \sum_{\mu \in \text{wt}(V^{(J)})} d_{\mu}
e^{-\mu} = \sum_{i=1}^{\infty} d_{\mu_i} e^{-\mu_i}.$$ Then the
twisted denominator identity yields
$$\prod_{\alpha \in \Delta_{+}(J)} (1- e^{-\alpha})^{\text{dim}
\g_{\alpha}} = 1 - \sum_{i=1}^{\infty} d_{\mu_i} e^{-\mu_i}.$$
Taking the logarithm of the left-hand side, we obtain
$$
\begin{aligned}
& \text{log}\left(\prod_{\alpha \in \Delta_{+}(J)} (1-
e^{-\alpha})^{\text{dim} \g_{\alpha}}\right) = \sum_{\alpha \in
\Delta_{+}(J)} \text{dim} \g_{\alpha} \, \text{log} (1-e^{-\alpha})
\\
&=\sum_{\alpha \in \Delta_{+}(J)} \text{dim} \g_{\alpha} \,
\sum_{k=1}^{\infty} \left(- \dfrac{1}{k}\, e^{-k \alpha} \right)
=-\sum_{\alpha \in \Delta_{+}(J)}
\sum_{k=1}^{\infty}\left(\dfrac{1}{k}\, \text{dim} \g_{\alpha}
\right) e^{-k \alpha}.
\end{aligned}
$$

On the other hand, taking the logarithm of the right-hand side, we
get
$$
\begin{aligned}
& \text{log}  \left(1-  \sum_{i=1}^{\infty} d_{\mu_i}
e^{-\mu_i}\right) = - \sum_{n=1}^{\infty} \dfrac{1}{n}
\left(\sum_{i=1}^{\infty} d_{\mu_i} e^{-\mu_i}\right)^n \\
& = - \sum_{n=1}^{\infty} \dfrac{1}{n} \left(\sum_{\substack{s=(s_i)
\\ \sum s_i =n}} \dfrac{n!}{\prod s_i !} \prod (d_{\mu_i} e^{-\mu_i})^{s_i}\right) \\
& = - \sum_{\beta} \left(\sum_{\substack{s=(s_i) \\ \sum s_i \mu_i =
\beta}} \dfrac{(|s|-1)!}{s!} \prod d_{\mu_i}^{s_i} \right)
e^{-\beta}=-\sum_{\beta} W^{(J)}(\beta) e^{-\beta}.
\end{aligned}
$$
Thus we have
$$\sum_{\alpha \in \Delta_{+}(J)} \sum_{k=1}^{\infty}
\left(\dfrac{1}{k}\, \text{dim} \g_{\alpha} \right) e^{-k \alpha} =
\sum_{\beta} W^{(J)}(\beta) e^{-\beta},$$ which yields
$$W^{(J)}(\beta) = \sum_{\beta = k \alpha} \dfrac{1}{k} \, \text{dim}
\g_{\alpha} = \sum_{d | \beta} \dfrac{1}{d}\, \text{dim} \g_{\beta /
d}.$$ Hence by M\"obius inversion, we obtain
$$\text{dim} \g_{\alpha} = \sum_{d | \alpha} \dfrac{1}{d}\, \mu(d)
\, W^{(J)}(\alpha / d).$$
\end{proof}

\vskip 3mm

\noindent {\bf Remark.} A natural interpretation of the twisted
denominator identity would be the Euler-Poincar\'e principle\,:
\begin{equation*}
\sum_{k=0}^{\infty} (-1)^k \text{ch} (\Lambda^k(\g_{-}^{(J)})
\otimes V(\lambda)) = \sum_{k=0}^{\infty} (-1)^k \text{ch} H_{k}
(\g_{-}^{(J)}, V(\lambda)),
\end{equation*}
where $V(\lambda)$ denotes the irreducible highest weight
$\g$-module with highest weight $\lambda \in P^{+}$. We expect the
following {\it Kostant-type homology formula} holds\,:
\begin{equation*}
H_{k}(\g_{-}^{(J)}, V(\lambda)) \cong \bigoplus_{\substack{w \in
W(J)\\ s \in F_{0} \\ l(w) + d(s) =k}} \epsilon(s^0)\,
V_{J}(w(\lambda + \rho-s)-\rho).
\end{equation*}

\vskip 3mm

We now consider some concrete applications of the root multiplicity
formula \eqref{eq:mult}.

\vskip 3mm

\begin{example} \label{ex:mult1} \hfill

\vskip 2mm

{\rm  Let $i \in I^{\text{im}}$ and let $\g_{(i)}$ be the
Borcherds-Bozec algebra associated with $A=(a_{ii})$. In this case,
we have
$$\Delta_{+}=\{k \alpha_i \mid k \ge 1 \}, \quad F_{0} = \{ s= k
\alpha_i \mid k \ge 0\}, \quad J=\emptyset, \quad W(J)=\{1\}.$$

\vskip 2mm

(a) If $a_{ii}=0$, for $s=k \alpha_i$ $(k \ge 0)$, we have
$$s = s^0, \quad d(s)=k, \quad \epsilon(s)=\phi(k),$$
which implies
$$V:= V^{(J)} = \bigoplus_{k=1}^{\infty}(-\phi(k)) \C u_{k},$$
where $u_k$ is a vector with weight $-k \alpha_i$ $(k \ge 1)$.

\vskip 3mm

Let $q=e^{-\alpha_i}$ and $\g(n)=\g_{n \alpha_i}$. Then the
denominator identity yields
$$\prod_{n=1}^{\infty} (1-q^n)^{\dim \g(n)} = 1 -
\sum_{k=1}^{\infty} (-\phi(k)) q^k = \prod_{n=1}^{\infty} (1-q^n),$$
from which we conclude
$$\dim \g(n) = 1 \quad \text{for all} \ \ n\ge 1$$
as we have seen in Example \ref{ex:char-rank2} (a).

\vskip 3mm

Let us compute $\dim \g(n)$ using the multiplicity formula
\eqref{eq:mult}, which will be much more complicated than the direct
calculation in this case. But the multiplicity formula
\eqref{eq:mult} can be applied to all Borcherds-Bozec algebras.

\vskip 3mm

We identify $\text{wt} (V)$ with $\Z_{>0}=\{1,2, \ldots \}$. Note
that $\dim V_{k}=-\phi(k)$ for all $k \ge 1$. For each $n \ge 1$,
set
$${\mathscr P}(n):= \{ s=(s_i)_{i \ge 1} \mid \sum i s_i =n \}$$
be the set of all partitions of $n$ into a sum of positive integers.
Then the Witt partition function is given by
$$W(n) = \sum_{s \in {\mathscr P}(n)} \dfrac{(|s|-1)!}{s!} \prod
(-\phi(i))^{s_i}.$$

Hence the root multiplicity can be computed as follows (and we
obtain a combinatorial identity as well)
$$\dim \g(n) = \sum_{d|n} \frac{1}{d} \mu(d) \sum_{s \in {\mathscr
P}(n/d)} \dfrac{(|s|-1)!}{s!} \prod (-\phi(-i))^{s_i},$$ which must
be equal to $1$ for all $n \ge 1$.

\vskip 3mm

For example, for $n=4$, the formula \eqref{eq:mult} gives
$$\dim \g(4) = W(4) - \frac{1}{2} W(2).$$
Note that
$$\prod_{n=1}^{\infty} (1-q^n) = 1 - q - q^2 + q^5 + \cdots.$$
Consider the partitions
\begin{equation*}
\begin{aligned}
& \  4=3+1=2+2=2+1+1=1+1+1+1,\\
& \ 2=1+1.
\end{aligned}
\end{equation*}
Then we have
\begin{equation*}
\begin{aligned}
W(4) = & (-\phi(1))+(-\phi(3))(-\phi(1)) + \frac{1}{2!} (-\phi(2))^2
\\
&  + \frac{2!}{2!} (-\phi(1))^2 (-\phi(2)) + \frac{3!}{4!}
(-\phi(1))^4
= \frac{7}{4}, \\
W(2) = & (-\phi(2)) + \frac{1}{2!} (-\phi(1))^2 = \frac{3}{2},
\end{aligned}
\end{equation*}
which gives
$$\dim \g(4) = \frac{7}{4} - \frac{1}{2} \times \frac{3}{2} =1$$
as expected.

\vskip 3mm

(b) If $a_{ii} <0$, for $s=k \alpha_i$ $(k \ge 1)$, we have
$$s=s^+, \quad  d(s)=1, \quad \epsilon(s)=-1, \quad d(0)=0,
\quad\epsilon(0)=1,$$ which implies
$$V:= V^{(J)} = \bigoplus_{k=1}^{\infty} \C u_{k},$$
where $u_k$ is a vector with weight $-k \alpha_i$ $(k \ge 1)$.

\vskip 3mm

Then the denominator identity is equal to
$$\prod_{n=1}^{\infty} (1 - q^n)^{\dim \g(n)} = 1 -
\sum_{k=1}^{\infty} q^k = 1 - q - q^2 - \cdots.$$ Since
$\text{wt}(V)=\{1, 2, 3, \ldots \}$ and $\dim V_{k}=1$ for all $k
\ge 1$, we have
$$W(n) = \sum_{s \in {\mathscr P}(n)} \dfrac{(|s|-1)!}{s!}$$
and hence
$$\dim \g(n) = \sum_{d|n} \frac{1}{d} \mu(d) \sum_{s \in {\mathscr P}(n/d)}
\dfrac{(|s|-1)!}{s!}.$$

\vskip 2mm

Let $n=4$. It is easy to compute $W(4)=15/4$, $W(2)=3/2$, from which
we obtain
$$\dim \g(4) = \frac{15}{4} - \frac{1}{2} \times \frac{3}{2} = 3.$$

\vskip 2mm

On the other hand, recall that we have
$$\prod_{n=1}^{\infty}
(1-q^n)^{\dim \g(n)} = 1-\sum_{k=1}^{\infty} q^k =
\dfrac{1-2q}{1-q}.$$

\vskip 2mm \noindent Let ${\mathscr L} = \bigoplus_{n=1}^{\infty}
{\mathscr L}_{n}$ be the free Lie algebra on $k$ generators. Then
$$\text{ch} U({\mathscr L}) = \dfrac{1}{1-kq} = \prod_{n=1}^{\infty}
(1 - q^n)^{-\dim {\mathscr L}_{n}},$$ and the $\dim {\mathscr
L}_{n}$ is given by the $k$-th Necklace polynomial of degree $n$
$$\dim {\mathscr L}_{n} = N(k,n):=\dfrac{1}{n} \sum_{d|n} \mu(d)
k^{n/d}.$$ Thus
\begin{equation*}
\begin{aligned}
\text{ch}U(\g)& =\prod_{n=1}^{\infty} (1-q^n)^{-\dim \g(n)}=
\dfrac{1-q}{1-2q} \\
& = (1-q) \prod_{n=1}^{\infty} (1-q^n)^{-N(2,n)} \\
&=(1-q)^{-1} \prod_{n=2}^{\infty} (1-q^n)^{-N(2,n)}.
\end{aligned}
\end{equation*}
Therefore we obtain
$$
\dim \g(n) = \begin{cases} 1 \ \ & \text{if} \ n=1, \\
N(2,n) \ \ & \text{if} \  n \ge 2. \end{cases}
$$

\vskip 2mm

For example, when $n=4$, we have
$$\dim \g(4) = N(2,4) = \frac{1}{4}(2^4 - 2^2) =3$$
as expected.

\vskip 3mm

More generally, we obtain the following interesting combinatorial identity

$$\sum_{d|n} \frac{1}{d} \mu(d) \sum_{s \in {\mathscr P}(n/d)}
\dfrac{(|s|-1)!}{s!} = \frac{1}{n} \sum_{d|n} \mu(d) 2^{n/d}.
$$ \qed}
\end{example}

\vskip 3mm

\begin{example} \label{ex:mult2a} \hfill

\vskip 2mm

{\rm

Let $I = \{0,1\}$, $A= \left(\begin{matrix}  2 & -a \\ -a & -2
\end{matrix} \right)$ $(a \ge 1)$ and let $\g$ be the Borcherds-Bozec
algebra associated with $A$.

\vskip 2mm

Then $I^{\text{re}}=\{0\}$, $I^{\text{im}}=\{1\}$, $W=\{1, r_0\}$ is
the Weyl group and $F_{0}=\{s=k \alpha_{1} \mid k \ge 0 \}$.
Moreover, we have
$$d(0)=1, \ \ \epsilon(0)=1, \ \ d(k \alpha_1) =1, \ \ \epsilon(k
\alpha_1) = -1 \ \ \text{for all} \ \ k \ge 1.$$

\vskip 2mm

Take $J=\{0\}$. Then $W(J)=\{1\}$ and
$$\g_{0}^{(J)}= \langle e_0, f_0, h_0 \rangle \oplus \C h_1 \cong
sl(2, \C) \oplus \C h_1.$$

\vskip 2mm

\noindent Note that, since $l(w) + d(s) >0$, $s$ must have the form
$s=k \alpha_1$ with $k \ge 1$. It follows that
\begin{equation*}
V:=V^{(J)} = \bigoplus_{k=1}^{\infty} V_{J}(-k \alpha_1),
\end{equation*}
where $V_{J}(-k \alpha_1)$ is the irreducible $sl(2, \C)$-module
with highest weight $-k \alpha_1$. Since $\langle h_0, -k \alpha_1
\rangle = ka$, we have
$$\text{wt}(V) = \bigcup_{k=1}^{\infty} \{-k \alpha_1 - l \alpha_0 \mid
0 \le l \le k a\} = \bigcup_{k=1}^{\infty} \{(k,l) \mid 0 \le l \le
ka \}$$ with $\dim V_{(k,l)}=1$ for all $k,l$.

\vskip 3mm

Give a total ordering on $\text{wt}(V)$ by
$$(k,l)<(p,q) \quad \text{if and only if} \ \ k<p \ \ \text{or} \ \
k =p, \, l<q.$$

\noindent  For each weight $(m,n) \in \Z_{>0} \times \Z_{>0}$, let
$${\mathscr P}(m,n) = \{s=(s_{ij}) \mid s_{ij} \in \Z_{\ge 0}, \ (i,j)
\in \text{wt}(V), \ \sum s_{ij}(i,j) =(m,n) \}$$ be the set of
partitions of $(m,n)$ into a sum of elements in $\text{wt}(V)$ with
respect to the total ordering given above. Then we obtain
\begin{equation*}
\begin{aligned}
& W(m,n) = \sum_{s \in {\mathscr P}(m,n)} \dfrac{(|s|-1)!}{s!}, \\
& \dim \g(m,n) = \sum_{d|(m,n)} \frac{1}{d} \mu(d) \sum_{s \in
{\mathscr P}(m/d, n/d)} \dfrac{(|s|-1)!}{s!}.
\end{aligned}
\end{equation*}

\vskip 3mm

For example, take $a=2$ and $\alpha=(4,4)$. Then the total ordering
of $\text{wt}(V)$ defined above is given by
\begin{equation*}
\begin{aligned}
& (1,0), (1,1), 1,2), \\
& (2,0), (2,1), (2,2), (2,3), (2,4), \\
& (3,0), (3,1), (3,2), (3,3), (3,4), \ldots \\
& (4,0), (4,1), (4,2), (4,3), (4,4), \ldots \\
& \qquad \qquad \cdots \cdots
\end{aligned}
\end{equation*}

\vskip 2mm

\noindent Consider the partitions
\begin{equation*}
\begin{aligned}
(4,4)& = (3,4) + (1,0) = (3,3) + (1,1) = (3,2) + (1,2) = (2,4) +
(1,0) + (1,0) \\
& = (2,4) +(2,0) = (2,3) + (2,1) = (2,3) + (1,1) + (1,0) = (2,2) +
(2,2) \\
& = (2,2) + (1,2) + (1,0) = (2,2) + (1,1) + (1,1) = (2,1) + (1,2) +
(1,1) \\
&=(2,0) + (1,2) + (1,2) = (1,2) + (1,2) + (1,0) + (1,0) \\
& = (1,2) + (1,1) + (1,1) + (1,0) = (1,1,) + (1,1) + (1,1) + (1,1),
\end{aligned}
\end{equation*}
which yield
\begin{equation*}
\begin{aligned}
W(4,4) &= 1 + 1 + 1 + 1 + \frac{2!}{2!} + 1 + 1 + 2 + \frac{1}{2!} +
2! + \frac{2!}{2!} + 2! \\
&  + \frac{2!}{2!} + \frac{3!}{2!2!} + \frac{3!}{2!} + \frac{3!}{4!}
\, = 20 + \frac{1}{4}.
\end{aligned}
\end{equation*}

\vskip 2mm

Next, the partitions
$$(2,2) = (1,2) + (1,0) = (1,1) + (1,1)$$
give
$$W(2,2) = 1 + 1 + \frac{1}{2!} = 2 + \frac{1}{2}.$$

Hence we obtain
$$\dim \g(4,4) = W(4,4) - \frac{1}{2} W(2,2) = 19.$$
\qed }
\end{example}

\vskip 3mm

\begin{example} \label{ex:mult2b} \hfill

\vskip 2mm

{\rm  Let $I = \{0,1\}$. $A=\left(\begin{matrix} 2 & -a \\
-a & 0
\end{matrix} \right)$ $(a \ge 1)$ and $\g$ be the Borcherds-Bozec
algebra associated with $A$.

\vskip 2mm

Then $I^{\text{re}}=\{0\}$, $I^{\text{im}}=\{1\}$, $W=\{1, r_0\}$ is
the Weyl group and $F_{0}=\{s=k \alpha_{1} \mid k \ge 0 \}$.
Moreover, we have
$$d(k \alpha_1) =k, \ \ \epsilon(k
\alpha_1) = \phi(k) \ \ \text{for all} \ \ k \ge 0.$$

\vskip 2mm

Take $J=\{0\}$. Then $W(J)=\{1\}$ and
$$\g_{0}^{(J)}= \langle e_0, f_0, h_0 \rangle \oplus \C h_1 \cong
sl(2, \C) \oplus \C h_1.$$

\vskip 2mm

\noindent Hence $s$ must have the form $s=k \alpha_1$ with $k \ge 1$
and
\begin{equation*}
V:=V^{(J)} = \bigoplus_{k=1}^{\infty} (-\phi(k)) V_{J}(-k \alpha_1),
\end{equation*}
where
$$\text{wt}(V) = \bigcup_{k=1}^{\infty} \{-k \alpha_1 - l \alpha_0 \mid
0 \le l \le k a\} = \bigcup_{k=1}^{\infty} \{(k,l) \mid 0 \le l \le
ka \}$$ with $\dim V_{(k,l)}=-\phi(k)$ for all $k,l$.

\vskip 2mm
\noindent  Therefore, we have
\begin{equation*}
\begin{aligned}
& W(m,n) = \sum_{s \in {\mathscr P}(m,n)} \dfrac{(|s|-1)!}{s!} \prod_{k=1}^{\infty} \prod_{l=1}^{ka} (-\phi(k))^{s_{kl}}, \\
& \dim \g(m,n) = \sum_{d|(m,n)} \frac{1}{d} \mu(d) \sum_{s \in
{\mathscr P}(m/d, n/d)} \dfrac{(|s|-1)!}{s!}\prod_{k=1}^{\infty}
\prod_{l=1}^{ka} (-\phi(k))^{s_{kl}}.
\end{aligned}
\end{equation*}

\vskip 3mm

For example, take $a=2$ and $\alpha=(4,4)$. Then using the
partitions in Example \ref{ex:mult2a}, we get
$$W(4,4) = 16 + \frac{1}{4}, \quad W(2,2) = 2 + \frac{1}{2},$$
which gives
$$\dim \g(4,4) = W(4,4) - \frac{1}{2} W(2,2) = 15.$$
\qed}
\end{example}

\vskip 7mm

\section{Monster Borcherds-Bozec algebra}

\vskip 3mm

Let $I$ be an index set and let $f:I \rightarrow \Z_{>0}$ be a
function such that $f(i)=1$ for all $i \in I^{\text{re}}$. Set
$${\widetilde I}:= \{(i,p) \mid i \in I, \ 1 \le p \le f(i) \}
\subset I \times \Z_{>0}.$$

\vskip 3mm

Let ${\widetilde A}$ be a Bocherds-Cartan matrix indexed by
${\widetilde I}$. Assume that for each pair $(i,j) \in I \times I$,
there is a block submatrix of size $f(i) \times f(j)$ in
${\widetilde A}$ such that all the entries of ${\widetilde A}$ in
this block submatrix are the same, say, $a_{ij}$. Then we obtain a
Borcherds-Cartan matrix $A=(a_{ij})_{i,j \in I}$. In this case, we
say that $A$ is the Borcherds-Cartan matrix with {\it charge}
${\mathbf f}=(f(i) \mid i \in I)$ induced from ${\widetilde A}$.
Conversely, given a Borcherds-Cartan matrix $A=(a_{ij})_{i, j \in
I}$ and a sequence ${\mathbf f}=(f(i) \mid i \in I)$, we obtain a
Borcherds-Cartan matrix ${\widetilde A} = ({\widetilde a}_{(i,p),
(j,q)})$ indexed by ${\widetilde I}$ by setting
$${\widetilde a}_{(i,p), (j,q)}=a_{ij} \quad \text{for} \ \ i, j \in I, \
1 \le p \le f(i), \ 1
\le q \le f(j).$$

\vskip 3mm

Hence we can develop the theory of Borcherds-Bozec algebras
associated with Borcherds-Cartan matrices $A$ with charge ${\mathbf
f}$ by taking the corresponding theory for the Borcherds-Bozec
algebras associated with ${\widetilde A}$.

\vskip 3mm

Let $I=\{-1\} \cup \{1,2,3, \ldots\}$ and let
\begin{equation} \label{eq:matrix-monster}
A=(a_{ij})_{i,j \in I} = (-(i+j))_{i,j \in I}
= \left(\begin{matrix}
\ 2 & \ 0 & -1 & -2 & \cdots \\
\ 0 & -2 & -3 & -4 & \cdots \\
-1 & -3 & -4 & -5 & \cdots \\
-2 & -4 & -5 & -6 & \cdots \\
\vdots & \vdots & \vdots & \vdots & \ddots
\end{matrix} \right)
\end{equation}
be the Borcherds-Cartan matrix with charge ${\mathbf c}=(c(i) \mid i
\in I)$, where $c(i)$ is the $i$-th coefficient of the elliptic
modular function
\begin{equation} \label{eq:j}
j(q) - 744 = \sum_{n=-1}^{\infty} c(n) q^n = q^{-1} + 196884 q +
21493760 q^2 + \cdots.
\end{equation}

\vskip 2mm

\noindent The associated Borcherds algebra $L$ is called the {\it
Monster Lie algebra} and played an important role in Borcherds'
proof of the Moonshine Conjecture \cite{Bor88, Bor92}. Note that
$\dim L_{\alpha_i} = c(i)$ for all $i \in I$. By setting
$\text{deg}\, \alpha_i = (1, i)$ for all $i \in I$, the Lie algebra
$L$ has a $(\Z \times \Z)$-grading
\begin{equation} \label{eq:Lie-monster}
L= \bigoplus_{m,n \in \Z} L_{(m,n)} \quad \text{such that} \ \ \dim
L_{(m,n)} = c(mn) \ \ \text{for all} \ \ m,n \ge 1.
\end{equation}

\vskip 2mm

On the other hand, for $m,n \ge 1$, let
\begin{equation}\label{eq:partition-monster}
{\mathscr P}(m,n) = \{s=(s_{ij})_{i,j \ge 1} \mid s_{ij} \in \Z_{\ge
0}, \ \sum s_{ij} (i,j) = (m,n) \}
\end{equation}
be the set of all partitions of $(m,n)$ into a sum of pairs of
positive integers with respect to any partial ordering. Then by the
root multiplicity formula for Borcherds algebras \cite{Kang94b}, the
Witt partition function is given by
\begin{equation} \label{eq:Witt-monster}
W(m,n) = \sum_{s \in {\mathscr P}(m,n)} \dfrac{(|s|-1)!}{s!}
\prod_{i,j} c(i+j-1)^{s_{ij}}
\end{equation}
and we obtain
\begin{equation} \label{eq:mult-monster}
\dim L_{(m,n)} =\sum_{d|(m,n)} \frac{1}{d} \, \mu(d)\sum_{s \in
{\mathscr P}(m/d,n/d)} \dfrac{(|s|-1)!}{s!} \prod_{i,j}
c(i+j-1)^{s_{ij}}.
\end{equation}

\vskip 2mm

Combined with \eqref{eq:Lie-monster}, we obtain a combinatorial
identity
\begin{equation} \label{eq:coeff-j}
c(mn) =\sum_{d|(m,n)} \frac{1}{d} \, \mu(d)\sum_{s \in {\mathscr
P}(m/d,n/d)} \dfrac{(|s|-1)!}{s!} \prod_{i,j} c(i+j-1)^{s_{ij}}.
\end{equation}

\vskip 2mm

\noindent In \cite{JLW}, it was noticed that, due to the identity
\eqref{eq:coeff-j}, the coefficients $c(1)$, $c(2)$, $c(3)$ and
$c(5)$ determine all other coefficients recursively.

\vskip 3mm

We now turn to the Borcherds-Bozec algebra ${\mathscr L}$ associated
with the Borcherds-Cartan matrix $A$ with charge ${\mathbf c}$ given
in \eqref{eq:matrix-monster} and \eqref{eq:j}, the {\it Monster
Borcherds-Bozec algebra}. We have $I^{\text{re}}=\{-1\}$,
$I^{\text{im}}=\{1,2,3, \ldots \}$ and
\begin{equation} \label{eq:F-monster}
F_{0} = \{0\} \cup \{l \alpha_k \mid k, l \ge 1 \}.
\end{equation}

\vskip 2mm

To apply our root multiplicity formula \eqref{eq:mult}, we take
$J=\{-1\}$. Then $W(J)=\{1\}$ and the condition $l(w) + d(s)>0$
implies $s=l \alpha_k$ with $k,l \ge 1$. Thus we have

\begin{equation} \label{eq:V-monster}
V=V^{(J)}=\bigoplus_{l=1}^{\infty} \bigoplus_{k=1}^{\infty} c(k)
V_{J}(-l \alpha_k),
\end{equation}
where $V_{J}(-l \alpha_k)$ is the irreducible $sl(2, \C)$-module
with highest weight $-l \alpha_k$. By the standard $sl(2,
\C)$-theory, since $\langle h_{-1}, -l \alpha_k \rangle = l(k-1)$,
we have
\begin{equation} \label{eq:wt-monster}
\begin{aligned}
\text{wt}(V) & = \bigcup_{l=1}^{\infty} \bigcup_{k=1}^{\infty} \{-l
\alpha_k -j \alpha_{-1} \mid 0 \le j \le l(k-1) \} \\
&= \bigcup_{l=1}^{\infty} \bigcup_{k=1}^{\infty} \{(l,k,j) \mid 0
\le j \le l(k-1) \}
\end{aligned}
\end{equation}
and
$$\dim V(l,k,j) = c(k) \quad \text{for all} \ \ k,l \ge 1, \ \ 0 \le
j \le l(k-1).$$

\vskip 2mm

Give a (lexicographic) total ordering on $\text{wt}(V)$ by $(l,k,j)
> (l',k',j')$ if and only if  (i) $l < l'$,  or (ii) $l=l'$, $k <
k'$,  or (iii) $l=l'$,  $k=k'$,  $j<j'$.

\vskip 3mm

For $\beta \in {\mathsf Q}_{+}$, let
\begin{equation} \label{eq:partition-bozec}
{\mathscr P}(\beta)=\{s=(s_{(l,k,j)}) \mid \sum s_{(l,k,j)} (l
\alpha_k + j \alpha_{-1}) = \beta \}
\end{equation}
be the set of all partitions of $\beta$ into a sum of elements in
$\text{wt}(V)$ with respect to the ordering given above. Since
$\text{mult}(l,k,j) = c(k)$ for all $k, l \ge 1$, $0 \le j \le
l(k-1)$, the Witt partition function is given by
\begin{equation} \label{eq:witt-bozec}
W(\beta) = \sum_{s \in {\mathscr P}(\beta)} \dfrac{(|s|-1)!}{s!}
\prod_{l,k,j} c(k)^{s_{(l,k,j)}}.
\end{equation}

\vskip 2mm

Therefore, for all $\alpha \in \Delta^{+}$, we have

\begin{equation} \label{eq:mult-bozec}
\dim {\mathscr L}_{\alpha} = \sum_{d |\alpha} \frac{1}{d} \, \mu(d)
\sum_{s \in {\mathscr P}(\alpha /d)} \dfrac{(|s|-1)!}{s!}
\prod_{l,k,j} c(k)^{s_{(l,k,j)}}.
\end{equation}

\vskip 3mm

\begin{example} \label{ex:monster-bozec} \hfill
{\rm

Let $\alpha = 2 \alpha_2 + 4 \alpha_1 + 2 \alpha_{-1}$. Then
$$\dim {\mathscr L}_{\alpha}= W(\alpha) - \frac{1}{2} W(\alpha
/2).$$ Note that the ordering on $\text{wt}(V)$ is given as follows.
$$\alpha_{1}, \ \alpha_2, \ \alpha_2 + \alpha_{-1}, \ 2 \alpha_1, \ 2
\alpha_2, \ 2 \alpha_2 + \alpha_{-1}, \ 2 \alpha_2 + 2 \alpha_{-1},
\cdots.$$

\vskip 2mm

To compute $W(\alpha)$ and $W(\alpha / 2)$, we consider the
partitions
\begin{equation*}
\begin{aligned}
\alpha & = 2 \alpha_2 + 4 \alpha_1 + 2 \alpha_{-1} \\
& = (2 \alpha_2 + 2 \alpha_{-1}) + 2 (2 \alpha_1) = (2 \alpha_2 + 2
\alpha_{-1}) + 4 (\alpha_1) \\
&= 2 (2 \alpha_1) + 2 (\alpha_2 + \alpha_{-1}) = (2 \alpha_1) + 2
(\alpha_2 + \alpha_{-1}) + 2 (\alpha_1) \\
&= 2 (\alpha_2 + \alpha_{-1}) + 4 (\alpha_1),\\
\alpha / 2 & = \alpha_2 + 2 \alpha_1 + \alpha_{-1} \\
& = (2 \alpha_1) + (\alpha_2 + \alpha_{-1})  = (\alpha_2 +
\alpha_{-1}) + 2 (\alpha_1),
\end{aligned}
\end{equation*}
which yield
\begin{equation*}
\begin{aligned}
W(\alpha) = & \frac{2!}{2!}\, c(2)\, c(1)^2 + \frac{4!}{4!}\, c(2)
\, c(1)^4
+ \frac{3!}{2!2!}\, c(2)^2 \, c(1)^2 \\
& + \frac{4!}{2!2!}\, c(1) \, c(2)^2 \, c(1)^2 + \frac{5!}{2! 4!} \,
c(2)^2 \,
c(1)^4 \\
= & \frac{5}{2} \, c(1)^4 \, c(2)^2 + c(1)^4 \, c(2) + 6 \, c(1)^3 \, c(2)^2 \\
& + \frac{3}{2} \, c(1)^2 \, c(2)^2 + c(1)^2 \, c(2), \\
W(\alpha / 2) = & c(1)\, c(2) + \frac{2!}{2!} \, c(2) \, c(1)^2 \\
= &c(1)^2 \, c(2) + c(1)\, c(2).
\end{aligned}
\end{equation*}

Hence we obtain
\begin{equation*}
\begin{aligned}
\dim {\mathscr L}_{\alpha} = & \frac{5}{2}\, c(1)^4 \, c(2)^2 +
c(1)^4 \,
c(2) + 6 \, c(1)^3 \, c(2)^2 \\
& + \frac{3}{2}\, c(1)^2 \, c(2)^2 + \frac{1}{2}\, c(1)^2 \, c(2) -
\frac{1}{2} \, c(1) \, c(2).
\end{aligned}
\end{equation*} \qed
}
\end{example}

\vskip 3mm

As in the case with the Monster Lie algebra, the Monster
Borcherds-Bozec algebra ${\mathscr L}$ has a $(\Z_{>0} \times
\Z_{>0})$-grading by setting $\text{deg}\, \alpha_i = (1, i)$ for $i
\in I$. Since $$\text{deg}\, (-l \alpha_k - j \alpha_{-1}) = - (l+j,
lk-j),$$ we have
\begin{equation} \label{eq:wtV-bozec}
\begin{aligned}
& \text{wt}(V)  = \bigcup_{l=1}^{\infty} \bigcup_{k=1}^{\infty}
\{(l+j, lk-j) \mid 0 \le j \le l(k-1) \} \\
&= \{(m,n) \mid m=l+j, \, n=lk-j, \ k,l \ge 1, \, 0 \le j \le l(k-1)
\}.
\end{aligned}
\end{equation}

\vskip 2mm

Each solution $(l,k,j)$ of the equation
$$m = l+j, \ n=lk-j, \ 0 \le j \le l(k-1)$$
has the contribution of $c(k) = c\left(\dfrac{m+n-l}{l}\right)$ to
the multiplicity of $(m,n)$ in $V$. Also the condition $0 \le
j=m-l-lk-n \le l(k-1)$ implies $1 \le l \le \text{min} (m,n)$.
Therefore, we obtain
$$\text{mult}(m,n) := d(m,n)  =
\sum_{l=1}^{\text{min}(m,n)} c\left(\dfrac{m+n-l}{l}\right),$$ where
$c(x)=0$ if $x$ is not an integer. Using the partition set
${\mathscr P}(m,n)$ in \eqref{eq:partition-monster}, the Witt
partition function is equal to
\begin{equation} \label{witt-bozec-a}
W(m,n) = \sum_{s \in {\mathscr P}(m,n)} \dfrac{(|s|-1)!}{s!}
\prod_{i,j} d(i,j)^{s_{ij}},
\end{equation}
and we obtain
\begin{equation} \label{eq:mult-bozec-a}
\dim {\mathscr L}_{(m,n)} = \sum_{d | (m,n)} \frac{1}{d} \mu(d)
\sum_{s \in {\mathscr P}(m/d, n/d)} \dfrac{(|s|-1)!}{s!} \prod_{i,j}
d(i,j)^{s_{ij}}.
\end{equation}

\vskip 3mm

\begin{example} \label{ex:bozec-a} \hfill

{\rm As an illustration, we will compute
$$\dim {\mathscr L}_{(4,2)}= W(4,2)- \frac{1}{2} W(2,1).$$

Using the partitions
$$(4,2) = (3,1)+(1,1) = (2,1) + (2,1),$$
we have $$W(4,2) = d(4,2) - d(3,1) \, d(1,1) + \frac{1}{2} \,
d(2,1)^2.$$ It is easy to see that
\begin{equation*}
d(1,1) = c(1), \ \ d(2,1) = c(2), \ \ d(3,1) = c(3), \ \ d(4,2) =
c(5) + c(2).
\end{equation*}
It follows that
$$W(4,2) = c(5)
+ c(3)\, c(1) + \frac{1}{2}\, c(2)^2 + c(2).$$

\vskip 2mm

Since $(2,1)$ has only one partition, we have
$$W(2,1) = d(2,1) = c(2).$$
Hence we obtain
\begin{equation*}
\dim {\mathscr L}_{(4,2)}  = c(5) + c(3)\, c(1) + \frac{1}{2} \,
c(2)^2 + \frac{1}{2}\, c(2).
\end{equation*}
Using the identity \eqref{eq:coeff-j}, we get a simpler expression
$$\dim {\mathscr L}_{(4,2)} = c(8) + c(2).$$
\qed}
\end{example}

\vskip 7mm

\section{Quiver varieties}

\vskip 3mm

Let $Q=(I, \Omega)$ be a locally finite quiver with loops, where $I$
is the set of vertices and $\Omega$ is the set of arrows. For an
arrow $h \in \Omega$, we denote by $\text{out}(h)$ (resp.
$\text{in}(h)$) the {\it outgoing vertex} (resp. {\it incoming
vertex}) of $h$. We often write $h: i \rightarrow j$ when
$\text{out}(h)=i$, $\text{in}(h) =j$.

\vskip 3mm

For $i \in I$, let $g_i$ be the number of loops at $i$ and for $i
\neq j$, let $c_{ij}$ denote the number of arrows in $\Omega$ from
$i$ to $j$. We denote by $\Omega(i)$ the set of loops at $i$. Define
a matrix $A_{Q}=(a_{ij})_{i,j \in I}$ by
\begin{equation} \label{eq:A_Q}
a_{ij} := \begin{cases} 2 - 2g_i \ \ & \text{if} \ i=j,\\
-c_{ij}-c_{ji} \ \ & \text{if} \ i \neq j.
\end{cases}
\end{equation}

\vskip 2mm

\noindent Then $A_{Q}$ is a {\it symmetric} Borcherds-Cartan matrix.
Note that $I^{\text{re}}= \{ i \in I \mid g_i =0 \}$, $I^{\text{im}}
= \{i \in I \mid g_{i} \ge 1 \}$. We will denote by $\g_{Q}$ the
Borcherds-Bozec algebra associated with $A_{Q}$.

\vskip 3mm

\begin{definition} \label{def:rep} \hfill

\vskip 2mm

{\rm
\begin{enumerate}
\item[(a)] A {\it representation} of $Q$ is a pair $(V,x)$, where
$V=\bigoplus_{i\in I} V_i$ is an $I$-graded vector space and
$x=(x_{h}:V_{\text{out}(h)} \rightarrow V_{\text{in}(h)})_{h \in
\Omega}$ is a family of linear maps such that

\begin{itemize}
\item[(i)] $V_{i}=0$ for all but finitely many $i \in I$,
\item[(ii)] $\dim V_i < \infty$ for all $i \in I$.
\end{itemize}

\vskip 2mm

\item[(b)] A {\it morphism} $\phi:(V,x) \rightarrow (W,y)$ of
representations consists of a collection of linear maps
$\phi=(\phi_i: V_i \rightarrow W_i)_{i \in I}$ such that for all $h
\in \Omega$ the following diagram is commutative.
\vskip 2mm
\begin{equation} \label{diagram:morphism}
\xymatrix{ V_{\text{out}(h)} \ar[d]^-{x_{h}}
\ar[r]^-{\phi_{\text{out(h)}}}&
W_{\text{out}(h)} \ar[d]^-{y_{h}} \\
V_{\text{in}(h)}  \ar[r]^-{\phi_{\text{in}(h)}} & W_{\text{in}(h)}}
\end{equation}
\end{enumerate}}
\end{definition}

\vskip 3mm

Let $V=\bigoplus_{i \in I} V_i$ be a representation of $Q$. We
define its {\it dimension vector} to be
$$\underline{\dim}\, V = \sum_{i \in I} (\dim V_i) \alpha_i \in {\mathsf Q}_{+}.$$

For an element $\alpha \in {\mathsf Q}_{+}$, fix a representation
$V$ with $\underline{\dim}\, V = \alpha$ and let
\begin{equation} \label{eq:repn variety}
E(\alpha):= \bigoplus_{h \in \Omega} \text{Hom} (V_{\text{out}(h)},
V_{\text{in}(h)}).
\end{equation}
\vskip 2mm

\noindent Then the group $G(\alpha):= \prod_{i \in I} GL(V_{i})$
acts on $E(\alpha)$ by conjugation

\begin{equation} \label{eq:G_V action}
(g \cdot x)_{h} = g_{\text{in}(h)} x_{h} g_{\text{out}(h)}^{-1}
\quad \text{for} \ \ g=(g_i)_{i \in I}, \, x=(x_h)_{h \in \Omega}.
\end{equation}

\vskip 2mm

\noindent We identify the set of isomorphism classes of
representations of $Q$ with dimension vector $\alpha$ with the set
of $G(\alpha)$-orbits in $E(\alpha)$.

\vskip 3mm

\begin{definition} \label{def:nilpotent}

\vskip 2mm {\rm

Let $x=(x_{h})_{h \in \Omega} \in E(\alpha)$.

\vskip 2mm
\begin{enumerate}
\item[(a)] $x$ is {\it nilpotent} if there exists an $I$-graded flag
$$L=(0=L_0 \subset L_1 \subset \cdots \subset L_r =V)$$  such
that $$x_{h}(L_k) \subset L_{k-1} \quad \text{for all} \ \ h \in
\Omega, \ 1 \le k \le r.$$

\vskip 2mm

\item[(b)] $x$ is {\it 1-nilpotent} if for each $i \in I^{\text{im}}$,
there exists a flag
$$L(i) = (0=L(i)_{0} \subset L(i)_{1} \subset \cdots \subset
L(i)_{t} = V_{i})$$  such that
$$x_{h}(L(i)_{k}) \subset L(i)_{k-1} \quad \text{for all} \ \ h \in
\Omega(i), \ 1 \le k \le t.$$
\end{enumerate}}
\end{definition}

\vskip 2mm

Set
\begin{equation*} 
E(\alpha)^{\text{nil}} := \{x \in E(\alpha) \mid \text{$x$ is
nilpotent} \}, \quad  E(\alpha)^{\text{1-nil}} := \{x \in E(\alpha)
\mid \text{$x$ is 1-nilpotent} \}.
\end{equation*}

\vskip 2mm

\noindent Then $E(\alpha)^{\text{nil}}$ and
$E(\alpha)^{\text{1-nil}}$ are Zariski-closed subvarieties of
$E(\alpha)$.

\vskip 3mm

Let $q$ be a power of some prime and let
$d_{\alpha}^{\text{nil}}(q)$ (respectively,
$d_{\alpha}^{\text{1-nil}}(q)$) denote the number of isomorphism
classes of nilpotent (respectively, 1-nilpotent) absolutely
indecomposable representations of $Q$ over $\F_{q}$ with dimensional
vector $\alpha$.

\vskip 3mm

\begin{proposition} \label{prop:Kac poly} \cite{BSV2016} \
{\rm Let $Q$ be a locally finite quiver with loops and let $\alpha
\in {\mathsf Q}_{+}$. Then the following statements hold.

\vskip 2mm

(a) There exist unique polynomials $A_{\alpha}^{\text{nil}}(t),
A_{\alpha}^{\text{1-nil}}(t) \in \Z[t]$ such that

\vskip 2mm
\begin{itemize}
\item[(i)] $A_{\alpha}^{\text{nil}}(q) =
d_{\alpha}^{\text{nil}}(q)$, $A_{\alpha}^{\text{1-nil}}(q) =
d_{\alpha}^{\text{1-nil}}(q)$ for all $q$.

\item[(ii)] $A_{\alpha}^{\text{nil}}(1) =
A_{\alpha}^{\text{1-nil}}(1)$.
\end{itemize}

\vskip 3mm

(b) $\dim (\g_{Q})_{\alpha}= A_{\alpha}^{\text{1-nil}}(0)$. }
\end{proposition}

\vskip 3mm

We will show that the set of positive roots of $\g_{Q}$ are in 1-1
correspondence with the set of dimension vectors of 1-nilpotent
indecomposable representations of $Q$. For this purpose, we briefly
recall some of important results on mixed Hodge polynomials and
$E$-polynomials (see \cite{HRV08} for more details).

\vskip 3mm

\begin{definition} \label{def:E-poly}

\vskip 2mm

{\rm  Let $X$ be a complex algebraic variety.

\vskip 2mm
\begin{enumerate}
\item[(a)] The {\it $E$-polynomial of $X$} is defined to be
$$E(X;x,y):=\sum_{p,q,j \ge 0} (-1)^j h_{c}^{p,q;j}(X) x^p y^q,$$
where $h_{c}^{p,q;j}(X)$ are compactly supported mixed Hodge
numbers.

\vskip 3mm

\item[(b)] A {\it spreading-out of $X$} is a separated scheme
$\widetilde{X}$ over a finitely generated $\Z$-algebra $R$ together
with an embedding $\varphi: R \hookrightarrow \C$ such that $X \cong
\widetilde{X}_{\varphi}$, the extension of scalars of
$\widetilde{X}$.

\vskip 3mm

\item[(c)] We say that $X$ has {\it polynomial count} if there exists a
polynomial $P_{X}(t) \in \Z[t]$ and a spreading-out $\widetilde{X}$
such that for every homomorphism $\phi:R \rightarrow \F_{q}$, we
have
$$P_{X}(q) = \# \{ \text{$\F(q)$-points of $\widetilde{X}_{\phi}$} \}.$$
\end{enumerate}}
\end{definition}

\vskip 3mm

\begin{theorem} \label{thm:Katz} \cite{HRV08} \
{\rm Suppose that a complex variety $X$ has  polynomial count with
the counting polynomial $P_{X}(t) \in \Z[t]$. Then we have
$$E(X;x,y) = P_{X}(xy).$$
}
\end{theorem}

\vskip 3mm

We apply Theorem \ref{thm:Katz} to our setting. Given an element
$\alpha \in {\mathsf Q}_{+}$, let $X(\alpha)$ be the variety of
isomorphism classes of 1-nilpotent indecomposable representations of
$Q$ over $\C$ with dimension vector $\alpha$. Then Proposition
\ref{prop:Kac poly} implies $X(\alpha)$ has polynomial count with
the counting polynomial $A_{\alpha}^{\text{1-nil}}(t)$. By Theorem
\ref{thm:Katz}, one can see that $X(\alpha) \neq \emptyset$ if and
only if $A_{\alpha}^{\text{1-nil}}(0) \neq 0$. Since $\dim
(\g_{Q})_{\alpha} = A_{\alpha}^{\text{1-nil}}(0)$, we conclude:

\vskip 3mm

\begin{proposition} \label{prop:pos roots} \
{\rm Let $Q$ be a locally finite quiver with loops and let $\g_{Q}$
be the Borcherds-Bozec algebra associated with $Q$. Then there is a
1-1 correspopndence
\begin{equation*}
\begin{aligned}
 \Delta_{+}(\g_{Q}) & =  \{\text{positive roots of $\g_{Q}$} \} \\
 \overset {\sim} \longleftrightarrow & \left\{
 \begin{aligned}
 & \text{dimension vectors of
1-nilpotent indecomposable} \\
&\text{representations of $Q$ over $\C$} \end{aligned}\right\}.
\end{aligned}
\end{equation*}
}
\end{proposition}

\vskip 3mm

\noindent
{\bf Remark.} The above argument is due to O. Schiffmann.

\vskip 8mm

\section{Constructible Functions}

\vskip 2mm

Let $X$ be a complex algebraic variety. A complex-valued function
$f:X \rightarrow \C$ is said to be {\it constructible} if \ (i)
$\text{Im} f$ is finite, \ (ii) $f^{-1}(c)$ is a constructible set
for all $c \in \C$. We define the {\it convolution integral} of $f$
over $X$ to be
\begin{equation} \label{eq:integral}
\int_{X} f d\chi := \sum_{c \in \C} c \, \chi(f^{-1}(c)),
\end{equation}
where $\chi$ denotes the Euler characteristic.

\vskip 3mm Let $Q=(I, \Omega)$ be a locally finite quiver with
loops. For each arrow $h \in \Omega$, we define the {\it opposite
arrow} $\overline{h}$ of $h$ by setting $\text{out}(\overline{h}) =
\text{in}(h)$, $\text{in}(\overline{h}) = \text{out}(h)$. Let
$\overline{\Omega}= \{\overline{h} \mid h \in \Omega \}$ and
$H=\Omega \cup \overline{\Omega}$. The quiver $\overline{Q} = (I,
H)$ thus obtained is called the {\it double quiver} of $Q$. Set
$\varepsilon(h) :=  1$ if $h \in \Omega$ and $\varepsilon(h) := -1$
if  $h \in \overline{\Omega}$.

\vskip 2mm

For $\alpha \in {\mathsf Q}_{+}$, fix an $I$-graded vector space
$V=\bigoplus_{i \in I} V_i$ with dimension vector $\alpha$. Set
\begin{equation*}
\overline{E}(\alpha):=\bigoplus_{h \in H} \text{Hom}
(V_{\text{out}(h)}, V_{\text{in}(h)}).
\end{equation*}
We have a symplectic form $\omega_{\alpha}$ on
$\overline{E}(\alpha)$ given by
\begin{equation*}
\omega_{\alpha}(x,y) = \sum_{h \in H} \text{tr}(\varepsilon(h) x_{h}
y_{\overline{h}}).
\end{equation*}
Note that $\omega_{\alpha}$ is preserved under the
$G(\alpha)$-action. We define the {\it moment map} by
\begin{equation*}
\begin{aligned}
\mu(\alpha) : \overline{E}(\alpha) & \longrightarrow \g(\alpha) =
\bigoplus_{i \in I} \text{End}(V_i) \\
x & \longmapsto \sum_{h \in H} \varepsilon(h) x_{\overline{h}} x_h.
\end{aligned}
\end{equation*}
Here, we identify $\g^* = \g$ via trace pairing.

\vskip 3mm

\begin{definition} \label{def:semi-nilpotent} \hfill

\vskip 2mm

{\rm
\begin{enumerate}
\item[(a)] $x=(x_h, x_{\overline{h}})_{h \in \Omega}$ is {\it nilpotent} if
there is an $I$-graded flag
$$L=(0=L_0 \subset L_1 \subset \cdots \subset L_r = V)$$ such that
$$x_h(L_k) \subset L_{k-1}, \ \ x_{\overline{h}}(L_{k}) \subset
L_{k-1} \ \ \text{for all} \ h \in \Omega, \ 1 \le k \le r.$$

\vskip 2mm

\item[(b)] $x=(x_h, x_{\overline{h}})_{h \in \Omega}$ is {\it
semi-nilpotent} if there is an $I$-graded flag
$$L=(0=L_{0} \subset L_{1} \subset \cdots \subset L_{r}=V)$$ such
that
$$x_h(L_k) \subset L_{k-1}, \ \ x_{\overline{h}}(L_{k}) \subset
L_{k} \ \ \text{for all} \ h \in \Omega, \ 1 \le k \le r.$$

\vskip 2mm

\item[(c)] $x=(x_h, x_{\overline{h}})_{h \in \Omega}$ is {\it strongly
semi-nilpotent} if there is an $I$-graded flag
$$L=(0=L_{0} \subset L_{1} \subset \cdots \subset L_{r}=V)$$ such
that

\begin{itemize}
\item[(i)] $L_{k}/L_{k-1}$ is concentrated on one vertex for each $1
\le k \le r$,

\item[(ii)] $x_h(L_k) \subset L_{k-1}$, $x_{\overline{h}}(L_{k}) \subset
L_{k}$ for all $h \in \Omega$, $1 \le k \le r.$
\end{itemize}
\end{enumerate}}
\end{definition}

\vskip 3mm

Set
\begin{equation*}
\Lambda(\alpha):=\{x\in \mu^{-1}(0) \mid \text{$x$ is strongly
semi-nilpotent} \}.
\end{equation*}

\vskip 3mm

\begin{proposition} \label{prop:Lambda} \cite{Bozec2014c, BSV2016} \
{\rm $\Lambda(\alpha)$ is a Lagrangian subvariety of
$\overline{E}(\alpha)$. }
\end{proposition}

\vskip 3mm

For each $\alpha \in {\mathsf Q}_{+}$, let ${\mathscr F}(\alpha)$ be
the space of constructible functions on $\Lambda(\alpha)$ which are
constant on $G(\alpha)$-orbits and set ${\mathscr F}:=
\bigoplus_{\alpha \in {\mathsf Q}_{+}} {\mathscr F}(\alpha)$.

\vskip 3mm

Let $x \in \Lambda(\alpha)$ and consider the constructible functions
$f \in {\mathscr F}(\beta)$, $g \in {\mathscr F}(\gamma)$ with
$\alpha = \beta + \gamma$. We define the {\it convolution product}
of $f$ and $g$ by
\begin{equation} \label{eq:product}
(f * g) (x) := \int_{W \subset V} f(x|_{W}) g(x|_{V/W}) d \chi,
\end{equation}
where the integral is taken over the projective variety of all
submodules $W$ of $V$ with dimension vector $\beta$ (cf.
\cite{Lus90, Lus93}). Then ${\mathscr F}$ becomes a ${\mathsf
Q}_{+}$-graded associative algebra under convolution product.

\vskip 3mm

For $(i,l) \in I^{\infty}$, the irreducible components of $\Lambda(l
\alpha_i)$ are parametrized by the compositions (resp. partitions)
of $l$ when $i \in I^{\text{im}} \setminus I^{\text{iso}}$ (resp. $i
\in I^{\text{iso}}$) \cite{Bozec2014b, Bozec2014c}. If $i \in
I^{\text{re}}$, $\Lambda(\alpha_i)$ is a point. Take the irreducible
component $Z_{i,(l)}$ of $\Lambda(l \alpha_i)$ corresponding to the
trivial composition (or partition) of $l$. Note that $Z_{i,(l)}$
consists of those $x$ such that $x_{h}=0$ for all $h \in \Omega(i)$.
In this case, there is no restriction on the linear transformations
$x_{h}$ with $h \in \overline{\Omega}(i)$. Hence any constructible
function on $Z_{i, (l)}$ can be considered as a constructible
function on the representation variety of $Q$ with dimension vector
$l \alpha_i$ such that $x_h=0$ for all $h \in \Omega(i)$. This fact
will be used in the proof of Proposition \ref{prop:main_b} without
any further explanation.

\vskip 3mm

Let $\theta_{i,l}$ be the characteristic function of $Z_{i,(l)}$ and
let ${\mathscr C}$ be the subalgebra of ${\mathscr F}$ generated by
$\theta_{i,l}$'s. Using the theory of perverse sheaves and crystal
bases for quantum Borcherds-Bozec algebras, Bozec proved:

\vskip 3mm

\begin{proposition} \label{prop:algC} \cite{Bozec2014b,
Bozec2014c} \
{\rm There is an algebra isomorphism
\begin{equation*}
U(\g_{Q}^{+})  \overset {\sim} \longrightarrow {\mathscr C} \quad
\text{given by} \ \  e_{il} \longmapsto \theta_{i,l} \ \ \text{for}
\ (i,l) \in I^{\infty}.
\end{equation*}
}
\end{proposition}

\vskip 8mm

\section{The Schofield construction}

\vskip 2mm

Let $E=\{S_{i,l} \mid (i,l) \in I^{\infty} \}$ be a set of variables
and let ${\mathscr E}=\C \langle E \rangle$ be the free associative
algebra on $E$ over $\C$. We will often write $S_i$ for $S_{i,1}$ $(i \in I^{\text{re}})$.

\vskip 3mm

\begin{definition} \label{def:flag} \
{\rm Let $Q=(I,\Omega)$ be a locally finite quiver with loops,
$(M=\bigoplus_{i\in I} M_i, \, x = (x_h)_{h \in \Omega})$ be a
representation of $Q$ and $w=S_{i_1, l_1} \cdots S_{i_r,l_r}$ be a
word in $E$. An $I$-graded filtration
$$L=(0=L_{0} \subset L_{1} \subset \cdots \subset L_{r}=M)$$
is called a {\it 1-nilpotent flag in $(M,x)$ of type $w$} if

\begin{itemize}
\item[(i)] $\underline{\dim}\, L_{k} / L_{k-1} = l_k \alpha_{i_k}$ for all $1 \le
k \le r$,

\item[(ii)] $x_h =0$ on $L_{k}/L_{k-1}$ for all $h \in \Omega(i_k)$,
$1 \le k \le r$.
\end{itemize}
}
\end{definition}

\vskip 3mm

For a word $w=S_{i_1, l_1} \cdots S_{i_r, l_r}$ in $E$ and a
representation $M$ of $Q$, we denote by $X_{M}(w)$ the projective
variety consisting of 1-nilpotent flags in $M$ of type $w$. We
define
\begin{equation} \label{eq:pairing_a}
\langle w, M \rangle := \chi(X_{M}(w)).
\end{equation}

\vskip 2mm

More generally, for $u=\sum_{w} a_w w \in {\mathscr E}$, we define
\begin{equation} \label{eq:pairing_b}
\langle u, M \rangle : = \sum_{w} a_w \langle w, M \rangle =\sum_{w}
a_w \, \chi(X_{M}(w)).
\end{equation}

\vskip 2mm

\noindent Set
\begin{equation}
{\mathscr I}:=\{ u \in {\mathscr E} \mid \langle u, M \rangle =0 \ \
\text{for all} \ M \}.
\end{equation}

\vskip 3mm

We will prove:

\begin{itemize}

\item[(a)] ${\mathscr I}$ is an ideal of ${\mathscr E}$.

\vskip 2mm

\item[(b)] ${\mathscr I}$ is a co-ideal of ${\mathscr E}$ under the
co-multiplication given by
$$\Delta: {\mathscr E} \rightarrow {\mathscr E} \otimes {\mathscr
E},  \quad S_{i,l} \mapsto S_{i,l} \otimes 1 + 1 \otimes S_{i,l}
\quad \text{for} \ \ (i,l) \in I^{\infty}.$$
\end{itemize}

\vskip 3mm

Then the {\it Schofield algebra} ${\mathscr R}:= {\mathscr E} \big /
{\mathscr I}$ would be a bi-algebra. To prove our claims, we recall
some basic facts on the Euler characteristic of algebraic varieties
(see, for example, \cite{Milnor and Stasheff, Sch, Spanier}).

\vskip 3mm

\begin{lemma} \label{lem:Euler} {\rm \
Let $X$ be a complex algebraic variety.

\vskip 2mm

\begin{enumerate}
\item[(a)] If $X=\bigsqcup_{\alpha} X_{\alpha}$ is a locally finite
partition by constructible subsets, then
$$\chi (X) = \sum_{\alpha} \chi(X_{\alpha}).$$

\vskip 2mm

\item[(b)] If $\pi: E \rightarrow X$ is a fibration with fiber $F$, then
$$\chi(E)=\chi(F) \chi(X).$$

\vskip 2mm

\item[(c)] Suppose $\C^*$ acts on $X$. If $U$ is an open subset on which
$\C^*$ acts almost freely; i.e., $\text{Stab}_{\C^*}(x)=\{1\}$ or
$\C^*$ for all $x \in U$, then $$ \chi(X) = \chi(X \setminus U).$$

\vskip 2mm

\item[(d)] If $\C^*$ acts on $X$, then
$$\chi(X)=\chi(X^{\C^*}),$$
where $$X^{\C^*}=\{x \in X \mid \lambda \cdot x = x \ \ \text{for
all} \ \lambda \in \C^*\}.$$
\end{enumerate}}
\end{lemma}

\vskip 3mm

\begin{proposition} \label{prop:ideal}

{\rm ${\mathscr I}$ is a two-sided ideal of ${\mathscr E}$.}
\end{proposition}

\vskip 2mm

\begin{proof} \ Let $u=\sum_{w} a_w w \in {\mathscr I}$.
We will show that $S_{i,l} \,u \in {\mathscr I}$, $u \, S_{i,l} \in
{\mathscr I}$ for all $(i,l)\in I^{\infty}$.

\vskip 3mm

Let $X_{M}^{+}(S_{i,l})$ be the variety of all submodules $M'$ of
$M$ with dimension vector $l \alpha_i$ such that $x_h(M') = 0$  for
all $h \in \Omega(i)$. For a word $w$ in $E$, we define a map
$$\phi_{w}: X_{M}(S_{i,l}\, w) \rightarrow X_{M}^{+}(S_{i,l})$$ by
$$L=(0=L_{0} \subset L_{1} \subset \cdots \subset L_{r}=M)
\mapsto L_{1}.$$

\vskip 2mm

Note that for all $M' \in X_{M}^{+}(S_{i,l})$, we have
$$\phi_{w}^{-1}(M') \cong X_{M/M'}(w),$$ which is independent of the
choice of $M'$. Thus we can find disjoint constructible subsets
$C_j$ of $X_{M}^{+}(S_{i,l})$ such that

\begin{itemize}
\item[(i)] $X_{M}^{+}(S_{i,l}) = \bigsqcup_{j} C_j$,

\item[(ii)] for all $M_{j}' \in C_j$ and for all $w$ with $a_w \neq 0
$, $\chi(\phi_{w}^{-1}(M_{j}'))$ is independent of the choice of
$M_{j}'$.
\end{itemize}

\vskip 3mm

Hence for any representation $M$ of $Q$, we have
\begin{equation*}
\begin{aligned}
\langle S_{i,l} \, u, M \rangle & = \sum_{w} a_{w} \langle S_{i,l}
\, w , M \rangle = \sum_{w} a_{w} \chi(X_{M}(S_{i,l}\, w)) \\
& = \sum_{w} a_{w} \sum_{j} \chi(C_{j}) \chi(\phi_{w}^{-1}(M_{j}'))
= \sum_{j} \chi(C_{j}) \sum_{w} a_{w} \chi(X_{M/M_{j}'}(w)) \\
& = \sum_{j} \chi(C_{j}) \sum_{w} a_{w} \langle w, M / M_{j}'
\rangle = \sum_{j} \chi(C_{j}) \langle u, M / M_{j}' \rangle =0,
\end{aligned}
\end{equation*}
which implies $S_{i,l} \, u \in {\mathscr I}$.

\vskip 3mm

Similarly, one can show that $u \, S_{i,l} \in {\mathscr I}$ for all
$(i,l) \in I^{\infty}$ using the map $$\psi_{w}: X_{M}(w
\,S_{i,l}) \rightarrow X_{M}^{-}(S_{i,l}),$$ where
$X_{M}^{-}(S_{i,l})$ is the variety of submodules $M''$ of $M$ such
that (i) $M/M''$ has dimension vector $l \alpha_i$, (ii) $x_h = 0$
on $M/M''$ for all $h \in \Omega(i)$.
\end{proof}

\vskip 3mm

Let ${\mathscr R}={\mathscr E} \big/ {\mathscr I}$ and set
$\text{deg}\, S_{i,l}:= l \alpha_i$. Since $X_{M}(w) = \emptyset$
unless $\underline{\dim}\, M = \text{deg}\, w$ and $M$ is
1-nilpotent, ${\mathscr I}$ is a ${\mathsf Q}_{+}$-graded ideal,
which implies ${\mathscr R}$ is a ${\mathsf Q}_{+}$-graded algebra.

\vskip 3mm

Let $E=\{S_{i,l} \mid (i,l) \in I^{\infty} \}$, $E'=\{S_{i,l}' \mid
(i,l) \in I^{\infty} \}$ and set $\widetilde{E} = E \cup E'$. For a
word $w$ in $\widetilde{E}$, we denote by $w_1$ (resp. $w_2$) the
word in $E$ (resp. $E'$) obtained from $w$. We will write $[w]=w_{1}
w_{2}$, the {\it normally ordered word} of $w$.

\vskip 3mm

\begin{lemma} \label{lem:comult_a} \
{\rm Let $M$ and $N$ be representations of $Q$ and let $w$ be a word
in $\widetilde{E}$. 
Consider $(M,N)$ as a representation of $Q$. 
Then we have
$$X_{(M,N)}(w) \cong X_{(M,N)}(w_1  w_2) \cong X_{M}(w_1) \times
X_{N}(w_2).$$

\noindent In particular,
$$\langle w, (M,N) \rangle = \langle w_1, M \rangle \, \langle w_2, N
\rangle. $$}
\end{lemma}

\begin{proof} \ \ Note that any 1-nilpotent flag in $(M,N)$ of type
$w$ is a mixture of a 1-nilpotent flag in $M$ of type $w_1$ and a
1-nilpotent flag in $N$ of type $w_2$. Hence the set of 1-nilpotent
flags in $(M,N)$ of type $w$ is in 1-1 correspondence with the set
of the products of a 1-nilpotent flag in $M$ of type $w_1$ and a
1-nilpotent flag in $N$ of type $w_2$, which yields

$$X_{(M,N)}(w) \cong X_{M}(w_1) \times X_{N}(w_2) \cong
X_{(M,N)}(w_1 w_2).$$

\vskip 2mm

\noindent It follows that
\begin{equation*}
\begin{aligned}
\langle w, (M,N) \rangle & = \chi(X_{(M,N)}(w)) = \chi(X_{M}(w_1)
\times X_{N}(w_2)) \\
&= \chi(X_{M}(w_1)) \chi(X_{N}(w_2)) = \langle w_1, M \rangle \,
\langle w_2, N \rangle,
\end{aligned}
\end{equation*}
which proves our claim.
\end{proof}

\vskip 3mm

Let $\widetilde{\mathscr E} = \Q \langle \widetilde{E} \rangle = \Q
\langle E \cup E' \rangle$, $\widetilde{\mathscr I}=\{u \in
\widetilde{\mathscr E} \mid \langle u, (M,N) \rangle =0 \ \text{for
all} \ M,N \}$ and set $\widetilde{\mathscr R}=\widetilde{\mathscr
E} \big/ \widetilde{\mathscr I}$.  We extend $[ \ \ ]$ by linearity
to obtain a linear map $[ \ \ ]: \widetilde{\mathscr E} \rightarrow
\widetilde{\mathscr E}$. We first prove:

\vskip 3mm

\begin{proposition} \label{prop:tilde R} \
{\rm There is an algebra isomorphism \ $\widetilde{\mathscr R} \cong
{\mathscr R} \otimes {\mathscr R}$. }
\end{proposition}

\begin{proof} \ \ For $u=\sum_w a_w w \in \widetilde{E}$, we have $[u] =
\sum_w a_w w_1 w_2$. Then by Lemma \ref{lem:comult_a}, we see that
$u=[u]$ in $\widetilde{\mathscr R}$. Therefore there exists a
surjective algebra homomorphism $$\phi: {\mathscr R} \otimes
{\mathscr R} \rightarrow \widetilde{\mathscr R}.$$

\vskip 2mm

We will prove $\phi$ is injective. For any $\beta \in {\mathsf
Q}_{+}$, define a matrix $K_{\beta}$ by
$$K_{\beta} =(\langle w, M \rangle)_{\substack{\text{deg}\, w = \beta
\\ \underline{\dim}\, M = \beta}}.$$

\noindent By definition, $\dim {\mathscr R}_{\beta} = \text{rank}
K_{\beta}$. Moreover, by Lemma \ref{lem:comult_a} again, we have
\begin{equation*}
\begin{aligned}
K_{\alpha} \otimes K_{\beta} & = (\langle u, M \rangle)_{\alpha}
\otimes (\langle w, N \rangle )_{\beta} = (\langle u, M \rangle
\langle w, N \rangle)_{(\alpha, \beta)} \\
& = (\langle uw, (M,N) \rangle)_{(\alpha, \beta)} =
\widetilde{K}_{(\alpha, \beta)}.
\end{aligned}
\end{equation*}

\noindent Hence
\begin{equation*}
\begin{aligned}
\dim {\widetilde{\mathscr R}}_{(\alpha, \beta)} &= \text{rank}
\widetilde{K}_{(\alpha, \beta)} = \text{rank} (K_{\alpha} \otimes
K_{\beta}) \\
& = (\text{rank} K_{\alpha})(\text{rank} K_{\beta}) = (\dim
{\mathscr R}_{\alpha}) (\dim {\mathscr R}_{\beta}) \\
&=\dim({\mathscr R} \otimes {\mathscr R})_{(\alpha, \beta)},
\end{aligned}
\end{equation*}
which implies $\widetilde{\mathscr R} \cong {\mathscr R} \otimes
{\mathscr R}$.  \end{proof}

\vskip 3mm

Define the algebra homomorphisms $\Delta: {\mathscr E} \rightarrow
{\mathscr E} \otimes {\mathscr E}$ and $\delta:  {\mathscr E}
\rightarrow \widetilde{\mathscr E}$ by
\begin{equation*}
\begin{aligned}
& \Delta: S_{i,l} \mapsto S_{i,l} \otimes 1 + 1 \otimes S_{i,l}, \\
& \delta: S_{i,l} \mapsto S_{i,l} + S_{i,l}' \quad \text{for} \ \
(i,l) \in I^{\infty}.
\end{aligned}
\end{equation*}

\vskip 3mm

Let $w = S_{i_1, l_1} \cdots S_{i_r, l_r}$ be a word in $E$. For a
subset $S \subset \{1, \ldots, r\}$, let $w_{S}$ be the word in
$\widetilde{E}$ obtained from $w$ by replacing $S_{i_s, l_s}$ by
$S_{i_s, l_s}'$ for $s \in S$. One can easily see that
$$\delta(w) = \sum_{S \subset \{1, \ldots, r\}} w_{S}.$$

\vskip 3mm

\begin{lemma} \label{lem:comult_b}

\vskip 2mm

{\rm For any word $w$ in $E$, we have
$$\langle w, M \oplus N \rangle = \langle \delta(w), (M,N)\rangle \
\ \text{for all} \ M,N.$$ }
\end{lemma}

\begin{proof} \ \ Let $w=w_1 \cdots w_r = S_{i_1, l_1} \cdots
S_{i_r, l_r}$ and $S \subset \{1, \cdots , r\}$. Then $w_{S}$ yields
a 1-nilpotent flag in $(M,N)$ of type $w_{S}$. Set
$$X_{(M,N)}(w_{S}) =\{\text{1-nilpotent flags in $(M,N)$ of type
$w_{S}$} \}.$$

\noindent  Note that we have
\begin{equation} \label{eq:delta}
\langle \delta(w), (M,N) \rangle = \sum_{S \subset \{1, \cdots, r
\}} \langle w_{S}, (M,N) \rangle = \sum_{S \subset \{1, \cdots, r
\}} \chi(X_{(M,N)}(w_{S})).
\end{equation}

\vskip 2mm

Each point of $X_{(M,N)}(w_{S})$ determines a sequence of
representations $L_{1}, \ldots, L_{r}$ such that $L_{j} \subset M$
if $j \notin S$ and $L_{j} \subset N$ if $j \in S$. Let $M_{j}$
(resp. $N_{j}$) be the largest submodule of $M$ (resp. $N$) in the
list $L_{1}, \ldots, L_{j}$ $(1 \le j \le r)$ or 0 if there is no
such submodule in the list. By setting $K_{j}:=M_{j} \oplus N_{j}$,
we obtain a 1-nilpotent flag in $M \oplus N$ of type $w$
$$K=(0=K_{0} \subset K_{1} \subset \cdots \subset K_{r} = M \oplus
N).$$ Thus for all $S \subset \{1, \ldots, r \}$, we see that
$X_{(M,N)}(w_{S})$ is a subvariety of $X_{M\oplus N}(w)$.

\vskip 3mm

Clearly, if $S \neq T$, $X_{(M,N)}(w_{S}) \cap X_{(M,N)}(w_{T}) =
\emptyset$. Therefore we have
\begin{equation*}
\begin{aligned}
& \bigsqcup_{S \subset \{1, \ldots, r \}} X_{(M,N)}(w_{S}) \subset
X_{M \oplus N}(w), \\
& \chi\left(\bigsqcup_{S \subset \{1, \cdots, r \}}
X_{(M,N)}(w_{S})\right) = \sum_{S \subset \{1, \ldots, r\}}
\chi(X_{(M,N)}(w_{S})).
\end{aligned}
\end{equation*}

\vskip 3mm

Conversely, given a 1-nilpotent flag $$K=(0=K_{0} \subset K_{1}
\subset \cdots \subset K_{r} =M \oplus N)$$ of type $w= S_{i_1, l_1}
\cdots S_{i_r, l_r}$ such that $K_{j}=M_{j} \oplus N_{j}$ with
$M_{j} \subset M$, $N_{j} \subset N$, we define a subset $S \subset
\{1, \ldots, r \}$ by setting  $j \in S$ if and only if $N_{j-1}
\subsetneq N_{j}$. We also define a sequence of submodules $L_{1},
\ldots, L_{r}$ by setting $L_{j}=N_{j}$ for $j \in S$, $L_{j}
=M_{j}$ for $j \notin S$. Thus we get a 1-nilpotent flag in $(M,N)$
of type $w_{S}$. Therefore we obtain a bijection
\begin{equation} \label{eq:bijection}
\bigsqcup_{S \subset \{1, \ldots, r \}} X_{(M,N)}(w_{S}) \overset
{\sim} \longleftrightarrow \left \{
\begin{aligned}
& \text{1-nilpotent flags in $M \oplus N$
of type $w$} \\
& \text{$K=(0=K_{0} \subset K_{1} \subset \cdots \subset K_{r}= M
\oplus N)$} \\
& \text{such that $K_j = (K_j \cap M) \oplus (K_j \cap N)$}
\end{aligned}
\right \}.
\end{equation}

\vskip 2mm

Note that $\C^*$ acts on $M \oplus N$ by
$$\lambda \cdot (m,n) = (\lambda m, n) \quad \text{for} \ \ \lambda
\in \C^*, m \in M, n \in N,$$ which induces a $\C^*$-action on the
Grassmannian variety  $Gr_{d}(M \oplus N)$ of $d$-dimensional
subspaces of $M \oplus N$. One can easily show that this
$\C^*$-action on $Gr_{d}(M \oplus N)$ is almost free. Moreover, $V
\subset M \oplus N$ is a fixed point under the $\C^*$-action if and
only if $V=(V \cap M) \oplus (V \cap N)$. By \eqref{eq:delta} and
\eqref{eq:bijection}, we see that $\C^*$ acts on $X_{M \oplus N}$
with
$$\left(X_{M \oplus N}(w)\right)^{\C^*} = \bigsqcup_{S \subset \{1,
\ldots, r\}} X_{(M,N)}(w_{S}).
$$

\noindent Therefore, by Lemma \ref{lem:Euler}, we conclude
\begin{equation*}
\begin{aligned}
& \langle w, M \oplus N \rangle = \chi(X_{M \oplus N}(w)) =
\chi((X_{M \oplus N}(w))^{\C^*}) \\
&\  = \chi\left(\bigsqcup_{S \subset \{1, \ldots, r\}}
X_{(M,N)}(w_{S})\right) = \sum_{S \subset \{1, \ldots, r\}}
\chi(X_{(M,N)}(w_{S})) \\
& \ =\sum_{S \subset \{1, \ldots, r\}} \langle w_{S}, (M,N) \rangle
= \langle \delta(w), (M,N) \rangle
\end{aligned}
\end{equation*}
as desired.
\end{proof}

\vskip 3mm

\begin{lemma} \label{lem:comult_c} \
{\rm \ Let $a \in E$ and write $\Delta(a) = \sum a_{(1)} \otimes
a_{(2)}$. For a word $w=S_{i_1, l_1} \cdots S_{i_r, l_r}$, set
$w'=S_{i_1, l_1}' \cdots S_{i_r, l_r}'$. Then the following
statements hold.

\vskip 2mm

(a) $[\delta(a)]=\sum a_{(1)} a_{(2)}'$.

\vskip 2mm

(b) For any representations $M$, $N$ of $Q$, we have
$$\langle a, M \oplus N \rangle = \sum \langle a_{(1)}, M \rangle
\langle a_{(2)}, N \rangle.$$ }
\end{lemma}

\begin{proof} \ \ We will prove (a) by induction on the length of
$a$. If $a=S_{i,l}$, our assertion is clear. Note that
$$
\begin{aligned}
\Delta(a S_{i,l}) & =\Delta(a) \Delta(S_{i,l})=(\sum a_{(1)}
\otimes a_{(2)}) (S_{i,l} \otimes 1 + 1 \otimes S_{i,l}) \\
&= \sum a_{(1)} S_{i,l} \otimes a_{(2)} + \sum a_{(1)} \otimes
a_{(2)} S_{i,l}.
\end{aligned}
$$

On the other hand,
\begin{equation*}
\begin{aligned}
& [\delta(a S_{i,l})]  = [\delta(a) \delta(S_{i,l})] =[(\sum
a_{(1)} a_{(2)}')(S_{i,l} + S_{i,l}')] \\
&=[\sum a_{(1)} a_{(2)}' S_{i,l}]+[\sum a_{(1)} a_{(2)}' S_{i,l}']
=\sum a_{(1)} S_{i,l} a_{(2)}' + \sum a_{(1)} a_{(2)}' S_{i,l}'
\end{aligned}
\end{equation*}
as desired.

\vskip 2mm

Thus by Lemma \ref{lem:comult_a} and Lemma \ref{lem:comult_b}, we
have
\begin{equation*}
\begin{aligned}
 \langle a,& \,  M\oplus N \rangle = \langle \delta(a), (M,N) \rangle =
\langle \delta(a)_{1} \delta(a)_{2}, (M,N) \rangle \\
& = \sum \langle a_{(1)} a_{(2)}', (M,N) \rangle =\sum \langle
a_{(1)}, M \rangle \langle a_{(2)}, N \rangle,
\end{aligned}
\end{equation*}
which proves (b).
\end{proof}

\vskip 3mm

\begin{proposition} \label{prop:co-ideal} \
{\rm \ ${\mathscr I}$ is a co-ideal of ${\mathscr E}$. That is,
$$\Delta({\mathscr I}) \subset {\mathscr E} \otimes {\mathscr I}
+  {\mathscr I} \otimes {\mathscr E}.$$ }
\end{proposition}

\begin{proof} \ If $a \in {\mathscr I}$, then
$\langle \delta(a), (M,N) \rangle = \langle a, M \oplus N \rangle
=0$ for all $M, N$, which implies $\delta(a) \in \widetilde{\mathscr
I}$. Thanks to the following commutative diagram
\begin{center}
\begin{tikzcd}[row sep=tiny]
                                                     & \widetilde{\mathscr E} \arrow[two heads]{r} & \widetilde{\mathscr{R}}\arrow{dd}{\wr}\\
{\mathscr E} \arrow{ur}{[\delta]}\arrow[dr, "\Delta" '] &          \\
                                                     & \mathscr{E} \otimes \mathscr{E} \arrow[two heads]{r}\arrow[uu] & \mathscr{R}\otimes \mathscr{R}
\end{tikzcd}
\end{center}
%
%
%
we conclude $$\Delta(a) \in {\mathscr E} \otimes {\mathscr I} +
{\mathscr I} \otimes {\mathscr E}.$$
\end{proof}

We say that an element $a \in {\mathscr R}$ is {\it primitive} if
$\Delta(a) = a \otimes 1 + 1 \otimes a$.

\vskip 3mm

\begin{theorem} \label{thm:primitive} \
{\rm \ Let ${\mathscr L} = \{a \in {\mathscr R} \mid \Delta(a) = a
\otimes 1 + 1 \otimes a \}$. Then the following statements hold.

\vskip 2mm

(a) ${\mathscr L}$ is a Lie bi-algebra.

\vskip 2mm

(b) ${\mathscr R} \cong U({\mathscr L})$. }
\end{theorem}

\begin{proof} \ Using Proposition \ref{prop:co-ideal}, it is straightforward to verify
(a). Since ${\mathscr R}$ is a bi-algebra generated by its primitive
elements, the statement (b) follows from \cite{Sweedler}.
\end{proof}

\vskip 3mm

\begin{proposition} \label{prop:primitive}

\vskip 2mm

{\rm Every primitive element vanishes on decomposable modules.
}
\end{proposition}

\begin{proof} \ Note that for any representation $M$ of $Q$,
$$\langle 1, M \rangle = \chi(X_{M}(1))=\chi(\emptyset) =0.$$ Hence
for any primitive element $a \in {\mathscr R}$ and representations $M$, $N$ of $Q$, we
have $$\langle a, M \oplus N \rangle = \langle a, M \rangle \langle
1, N \rangle + \langle 1, M \rangle \langle a, N \rangle =0,$$ which
is our assertion.
\end{proof}

\vskip 3mm

\begin{corollary} \label{cor:decomp} \
{\rm \ If there is no 1-nilpotent indecomposable representation of
$Q$ with dimension vector $\alpha$, then ${\mathscr L}_{\alpha}=0$.
}
\end{corollary}

\begin{proof} \ Recall that for $a \in {\mathscr L}_{\alpha}$,
$\langle a, M \rangle =0$ unless $\underline{\dim}\, M = \alpha$ and
$M$ is 1-nilpotent. Hence by Proposition \ref{prop:primitive},
$\langle a, M \rangle =0$ for all $M$, which implies $a=0$.
\end{proof}

\vskip 3mm

\begin{proposition} \label{prop:Serre_a} \
{\rm \ For $i \in I^{\text{re}}$, $(j,l) \in I^{\infty}$ with $i
\neq (j,l)$, we have
$$(\text{ad} S_{i})^{1- l a_{ij}} (S_{j,l}) =0 \ \ \text{in} \ \, {\mathscr L},$$
where $a_{ij} = -c_{ij} -c_{ji}$ as in \eqref{eq:A_Q}. }
\end{proposition}

\begin{proof} \ Let $u =(\text{ad} S_{i})^{1- l a_{ij}} (S_{j,l})
\in {\mathscr L}_{(1- l a_{ij})\alpha_i + l \alpha_j}$. Since
$$\langle h_i, -l a_{ij} \alpha_i + l \alpha_j \rangle = -2 l a_{ij}
+ l a_{ij} = -l a_{ij} \ge 0,$$ we see that
$$-l a_{ij} \alpha_i + l \alpha_j + \alpha_i = (1-l a_{ij}) \alpha_i
+ l \alpha_j \notin \Delta_{+}(\g_{Q}).$$ By Proposition
\ref{prop:pos roots}, there is no 1-nilpotent indecomposable
representation of $Q$ with dimension vector $(1-l a_{ij}) \alpha_i +
l \alpha_j$. Hence by Corollary \ref{cor:decomp}, $u=0$.
\end{proof}

\vskip 3mm

\begin{proposition} \label{prop:Serre_b} \
{\rm \ If $a_{ij}=0$ for some $i,j \in I$, then
$$[S_{i,k}, S_{j,l}]=0 \ \ \text{for all} \ k,l \ge 1.$$

}
\end{proposition}

\begin{proof} \ Let $u=[S_{i,k}, S_{j,l}] \in {\mathscr L}_{k
\alpha_i + l \alpha_j}$. Note that $(\g_{Q})_{k \alpha_i + l
\alpha_j}$ is spanned by the brackets of the form
$$[e_{i, p_1}, e_{i, p_2}, \cdots, e_{i, p_r}, e_{j,q_1}, e_{j,
q_2}, \cdots, e_{j, q_s}] \ \ (r,s \ge 1, \ p_m, q_n \ge 1).$$ But
if $a_{ij}=0$, then $[e_{i,p}, e_{j,q}]=0$ for all $p, q \ge 1$, and
hence all the above brackets are $0$, from which we conclude
$(\g_{Q})_{k \alpha_i + l \alpha_j}=0$. Thus by Proposition
\ref{prop:pos roots}, $k \alpha_i + l \alpha_j$ is not a dimension
vector of a 1-nilpotent indecomposable representation of $Q$ and
Corollary \ref{cor:decomp} implies $u=0$.
\end{proof}

\vskip 3mm

Combining Proposition \ref{prop:Serre_a} and Proposition
\ref{prop:Serre_b}, we obtain:

\vskip 2mm

\begin{theorem} \label{thm:main_a} \hfill

\vskip 2mm

{\rm (a) There is a surjective algebra homomorphism
$$\Phi:U^+(\g_{Q}) \longrightarrow {\mathscr R} \ \
\text{given by} \ \  e_{il} \mapsto S_{i,l} \ \ \text{for} \ (i,l)
\in I^{\infty}.$$

\vskip 2mm

(b) There is a surjective Lie algebra homomorphism
$$\Phi_{0}: \g_{Q}^{+} \longrightarrow {\mathscr L}
\ \ \text{given by} \ \ e_{il} \mapsto S_{i,l} \ \ \text{for} \
(i,l) \in I^{\infty}.$$ }
\end{theorem}

\vskip 3mm

We will prove $\Phi$ and $\Phi_{0}$ are isomorphisms. To this end,
it suffices to prove the following proposition.

\vskip 2mm

\begin{proposition} \label{prop:main_b} \
{\rm Let ${\mathscr C}$ be the algebra generated by the
characteristic functions $\theta_{i,l}$ $(i,l) \in I^{\infty}$. Then
there exists an algebra isomorphism
$${\mathscr R} \overset{\sim} \longrightarrow {\mathscr C} \ \
\text{given by} \ \ S_{i,l} \mapsto \theta_{i,l} \ \ \text{for} \
(i,l) \in I^{\infty}.$$ }
\end{proposition}

\begin{proof} \ Given a word $w=S_{i_1, l_1} \cdots S_{i_r, l_r}$ in
$E$, we denote $\theta_w := \theta_{i_1, l_1} * \cdots *
\theta_{i_r, l_r}$. Then for any representation $(M,x)$ of $Q$ with
dimension vector $\alpha = l_1 \alpha_{i_1} + \cdots + l_r
\alpha_{i_r}$, we have
\begin{equation*}
\begin{aligned}
& \theta_{w}(x) = \theta_{i_1, l_1} *  \cdots * \theta_{i_r, l_r} (x) \\
& = \int_{L=(0 \subset L_{1} \subset \cdots \subset L_{r}=M)}
\theta_{i_1, l_1}(x|_{L_1}) \theta_{i_2, l_2}(x|_{L_2/L_1}) \cdots
\theta_{i_r, l_r}(x|_{L_{r}/L_{r-1}}) d\chi\\
&= \chi\left( \left\{ L=(0  \subset L_{1} \subset \cdots \subset
L_{r}=M) \mid \begin{aligned} & {\rm (i)} \,\underline{\dim} \,
L_{k}/L_{k-1}
= l_k \alpha_{i_k},\\
& {\rm (ii)}\, x_h = 0 \ \ \text{on} \ L_{k}/L_{k-1} \\
& \quad \ \ \text{for all} \ h \in \Omega(i_k)
\end{aligned} \right\} \right) \\
&= \chi(X_{M}(w)) = \langle w, M \rangle .
\end{aligned}
\end{equation*}

\vskip 2mm

Therefore, whenever there is a relation $\sum_{w} a_{w} w =0$ in
${\mathscr R}$, we have the corresponding relation $\sum_{w} a_{w}
\theta_{w} =0$ in ${\mathscr C}$, and vice versa, which proves our
claim.
\end{proof}

\vskip 3mm

Now we obtain the {\it Schofield construction} of Borcherds-Bozec
algebras.

\begin{theorem} \label{thm:main_c} \
{\rm The homomorphisms $\Phi$ and $\Phi_{0}$ are isomorphisms.}
\end{theorem}

\begin{proof}
Our assertion follows from Proposition \ref{prop:algC}, Theorem
\ref{thm:main_a} and Proposition \ref{prop:main_b}.
\end{proof}

\end{document}